\documentclass[12pt,a4paper,reqno,oneside]{amsart}

\usepackage{graphicx,psfrag,amssymb,amscd}
\usepackage[vcentering]{geometry}
\newcommand{\field}[1]{\ensuremath{\mathbb{#1}}}
\newcommand{\integer}{\field{Z}}
\newcommand{\naturalnumber}{\field{N}}
\newcommand{\Z}[1]{\ensuremath{\integer/#1\integer}}
\newcommand{\rational}{\field{Q}}
\newcommand{\real}{\field{R}}

\newcommand{\complete}[2][n]{\ensuremath{K_{#2}^{#1}}}
\newcommand{\crossed}[1]{\ensuremath{C(#1)}}
\newcommand{\divided}[2][n]{\ensuremath{\Delta^{#1}(#2)}}
\newcommand{\dividedsigma}[2][n]{\ensuremath{\Sigma^{#1}(#2)}}
\newcommand{\comdiv}[3][n]{\ensuremath{\complete[#1]{#2}(#3)}}
\newcommand{\prism}{\ensuremath{\mathcal{P}}}
\newcommand{\cyl}{\ensuremath{\mathcal{C}}}
\newcommand{\sph}{\ensuremath{\mathcal{S}}}

\newcommand{\link}[2]{\ensuremath{\ell k(#1,#2)}}
\newcommand{\linktwo}[2]{\ensuremath{\ell k_2(#1,#2)}}
\newcommand{\ambient}[1][n]{\ensuremath{\real^{2#1+1}}}
\newcommand{\rthree}{\ensuremath{\real^3}}

\newcommand{\eps}{\varepsilon}
\DeclareMathOperator{\id}{id}
\newcommand{\key}[1][n]{\kappa_{#1}}

\newcommand{\vect}[1]{\ensuremath{\mathbf{#1}}}
\newcommand{\vectg}[1]{\ensuremath{\boldsymbol{#1}}} 

\newcommand{\citeramsey}{\cite{flapan2002,flapan-mellor-naimi2008,flapan-pommersheim-foisy-naimi2001,fleming2007,fleming-diesl2005}}
\newcommand{\etal}{~et~al.}

\newtheorem{theorem}{Theorem}[section]
\newtheorem{lemma}[theorem]{Lemma}
\newtheorem{corollary}[theorem]{Corollary}
\newtheorem*{lemma*}{Lemma}
\numberwithin{equation}{section}

\theoremstyle{definition}
\newtheorem{construction}[theorem]{Construction}

\theoremstyle{remark}
\newtheorem{remark}[theorem]{Remark}

\begin{document}

\title{Some Ramsey-type results on intrinsic linking of $n$-complexes}
\author{Christopher Tuffley}
\date{\today}
\address{Institute of Fundamental Sciences, Massey University,
         Private Bag 11 222, Palmerston North 4442, New Zealand}
\email{c.tuffley@massey.ac.nz}


\begin{abstract}
Define the \emph{complete $n$-complex on $N$ vertices}, \complete{N},
to be the $n$-skeleton of an $(N-1)$-simplex. We show that embeddings
of sufficiently large complete $n$-complexes in \ambient\ necessarily
exhibit complicated linking behaviour, thereby extending known results
on embeddings of large complete graphs in $\real^3$ (the case $n=1$)
to higher dimensions. In particular, we prove the existence of links
of the following types: $r$-component links, with the linking pattern
of a chain, necklace or keyring; 2-component links with linking number
at least $\lambda$ in absolute value; and $2$-component links with
linking number a non-zero multiple of a given integer $q$. For fixed
$n$ the number of vertices required for each of our results grows at
most polynomially with respect to the parameter $r$, $\lambda$ or $q$.
\end{abstract}

\maketitle

\section{Introduction}

In the 1980s Sachs~\cite{sachs83} and Conway and Gordon~\cite{conway-gordon83}
proved that an embedding of the complete graph $K_6$ in $\real^3$
necessarily contains a pair of disjoint cycles that form a non-split
link. This fact is expressed by saying that $K_6$ is 
\emph{intrinsically linked}. Conway and Gordon also showed that
every embedding of $K_7$ in $\real^3$ contains a cycle that forms
a nontrivial knot, and we say that $K_7$ is \emph{intrinsically knotted}.

Since these papers, the study of intrinsic knotting and linking has
been pursued in several directions, and we refer the reader to
Ram{\'\i}rez Alfons{\'\i}n~\cite{ramirezalfonsin-2005} for a survey of some known results.
One such direction is to show that embeddings of larger complete
graphs necessarily exhibit more complex knotting and linking
behaviour. Restricting our attention to linking,
Flapan\etal~\cite{flapan-pommersheim-foisy-naimi2001} and Fleming and
Diesl~\cite{fleming-diesl2005} have shown that embeddings of
sufficiently large complete graphs must contain non-split
$r$-component links; Flapan~\cite{flapan2002} has shown that they must
contain 2-component links with high linking number; and
Fleming~\cite{fleming2007} has extended work by Fleming and
Diesl~\cite{fleming-diesl2005} to show that, given an integer $q$,
they must contain 2-component links with linking number a nonzero
multiple of $q$.

We will refer to results such as those described above as
\emph{Ramsey-type results} on intrinsic linking. Perhaps the strongest
results in this direction are those of Negami~\cite{negami1998-good} and
Flapan, Mellor and Naimi~\cite{flapan-mellor-naimi2008}. Restricting
attention to embeddings with a projection that is a ``good drawing'',
Negami shows that, given a link $L$, for $n,m$ sufficiently large
every such embedding of the complete bipartite graph $K_{n.m}$
contains a link that is ambient isotopic to $L$.  The restriction to
embeddings with a projection that is a good drawing excludes local
knots in the edges, which is necessary but not sufficient
(Negami~\cite{negami2001}) for the result to hold.  With no
restriction on the embedding, Flapan, Mellor and Naimi show that
intrinsic knotting and linking are arbitrarily complex in the
following sense: Given positive integers $r$ and $\alpha$, embeddings
of sufficiently large complete graphs contain $r$-component links in
which the second co-efficient of the Conway polynomial of each
component, and the linking number of each pair of components, is at
least $\alpha$ in absolute value.

Extending the result of Sachs~\cite{sachs83} and Conway and
Gordon~\cite{conway-gordon83} in another direction, we may consider
embeddings of $n$-complexes in $\real^d$. By a general position
argument every $n$-complex embeds in $\real^{2n+1}$, and a pair of
disjoint $n$-spheres in $\real^{2n+1}$ have a well defined linking
number (the homology class of one component in the $n$th homology
group of the complement of the second, which is isomorphic to
$\integer$), so we take $d=2n+1$. Define the \emph{complete
  $n$-complex on $N$ vertices}, \complete{N}, to be the $n$-skeleton
of an $N-1$ simplex. Then Lov{\'a}sz and
Schrijver~\cite[Cor.~1.1]{lovasz-schrijver1998},
Taniyama~\cite{taniyama2000}, Melikhov~\cite[Ex.~4.7]{melikhov2009} and
Melikhov~\cite[Ex.~4.9]{melikhov2011} show by various arguments that
\complete{2n+4}\ is intrinsically linked, in the sense that every
embedding in \ambient\ contains a pair of disjoint $n$-spheres that
have nonzero linking number. Since $\complete[1]{N}\cong K_N$ this
specialises to the $K_6$ result in the case $n=1$.

The purpose of this paper is to establish some Ramsey-type results for
embeddings of complete $n$-complexes in \ambient.  Our results are
already known for embeddings of complete graphs in $\real^3$, and our
arguments will typically mimic the proof of the corresponding
1-dimensional result. However, in the case of Theorem~\ref{modq.th} we
will obtain a better bound for $n=1$ than that previously known; and
in addition, some constructions used in the arguments require
modifications in higher dimensions. These modifications are needed for
two main reasons: Firstly, $\partial D^n=S^{n-1}$ is disconnected for
$n=1$, but not for $n\geq 2$; and secondly, triangulations of $D^n$
have simpler combinatorics for $n=1$ than they do for $n\geq 2$.

We note that for $n\geq2$ an $n$-sphere does not
knot in \ambient\ for reasons of co-dimension, and an arbitrary 
$n$-complex does not necessarily embed in $\real^{n+2}$. Thus, we will not seek
to establish any results on intrinsic knotting of complete 
$n$-complexes.

\subsection{Statement of results}

In what follows, a \emph{$k$-component link} means $k$
disjoint $n$-spheres embedded in \ambient. Given a 2-component
link $L_1\cup L_2$ we will write \link{L_1}{L_2}\ for their linking number,
and $\linktwo{L_1}{L_2}$ for their linking number mod two.  For
$\{i,j\}=\{1,2\}$ the integral linking number is given by the homology
class $[L_i]$ in $H_n(\ambient-L_j;\integer)\cong\integer$.

Our first result is similar to Theorems~1 and~2 of
Flapan\etal~\cite{flapan-pommersheim-foisy-naimi2001}, and shows that
embeddings of sufficiently large complete $n$-complexes necessarily
contain non-split $r$-component links.  Moreover, the number of
vertices required grows at most linearly with respect to each of $r$
and $n$.

\begin{theorem}
\label{chains.th}
Let $r\in\naturalnumber$, $r\geq 2$. 
\begin{enumerate}
\renewcommand{\theenumi}{\alph{enumi}}
\item
\label{chain.item}
For $N\geq (2n+4)(r-1)$ every embedding of $\complete{N}$ in
$\ambient$ contains an
$r$-component link $L_1\cup L_2\cup\cdots\cup L_r$ such that 
\begin{equation}
\label{sequentiallinking.eq}
\linktwo{L_i}{L_{i+1}} \neq 0
\end{equation}
for $i=1,\ldots,r-1$.
\item
\label{necklace.item}
If $r\geq 3$ then for $N\geq (2n+4)r$ every embedding of
$\complete{N}$ in $\ambient$ contains an $r$-component link $L_1\cup
L_2\cup\cdots\cup L_r$ satisfying
equation~\eqref{sequentiallinking.eq} for $i=1,\ldots,r$ (subscripts
taken mod $r$).
\end{enumerate}
\end{theorem}

The link of Theorem~\ref{chains.th}(\ref{chain.item}) resembles a
chain, and the link of Theorem~\ref{chains.th}(\ref{necklace.item})
resembles a necklace, except that there is no requirement that
non-adjacent components do not also link. Our next result generalises
Lemma~2.2 of Fleming and Diesl~\cite{fleming-diesl2005}, and yields
links that resemble a bunch of keys on a keyring.  However, there is
again no requirement that the ``keys'' do not also link each other,
and following Flapan\etal~\cite{flapan-mellor-naimi2008} we call such
a link a \emph{generalised keyring}. Generalised keyrings will play a
crucial role in establishing our results for 2-component links, in
Theorems~\ref{linkingnumber.th}--\ref{modp.th}.

\begin{theorem}
\label{keyring.th}
For a natural number $r$ define
\[
\key(r)=4r^2(2n+4)+n+\left\lceil\frac{4r^2-2}{n}\right\rceil+1.
\]
Then every embedding
of \complete{\key(r)}\ in $\real^{2n+1}$ contains an $(r+1)$-component link
$R\cup L_1\cup L_2\cup\cdots\cup L_r$ such that
\[
\linktwo{R}{L_i}=1
\]
for $i=1,\ldots,r$. 
\end{theorem}
 
Observe that $\key(r)$ grows quadratically in $r$ and linearly in 
$n$. The existence of generalised keyrings in embeddings of
\complete{N}\ for $N$ sufficiently large may be established by
following Fleming and Diesl's argument, or that of 
Flapan\etal~\cite[Lemma~1]{flapan-mellor-naimi2008}; the 
Fleming-Diesl argument leads to a bound that grows exponentially with
respect to $r$, and so we will
follow the argument of Flapan\etal, as this leads to the polynomial
bound given above.
For $n=1$ the term $n+\lceil(4r^2-2)/n\rceil+1$ of $\key$ is not 
needed, so it suffices to take $\key[1](r)=24r^2$. This bound follows
from Flapan\etal~\cite[Lemma~1]{flapan-mellor-naimi2008}, although they
do not state the bound explicitly.

Our last three results concern linking number in 2-component
links. The first extends Theorem~2 of Flapan~\cite{flapan2002}
to higher dimensions (although our proof will be based on a
technique from Lemma~2 of Flapan\etal~\cite{flapan-mellor-naimi2008},
as this leads to a better bound in higher dimensions):
\begin{theorem}
\label{linkingnumber.th}
Let $\lambda\in\naturalnumber$ be given, and let
\[
N= \key({2\lambda-1})+n+\left\lceil\frac{2\lambda-1}{n}\right\rceil +1.
\]
Then
every embedding of \complete{N}\ in $\real^{2n+1}$ contains a 
two-component link $L\cup J$ such that, for some orientation of
the components, $\link{L}{J}\geq\lambda$. 
\end{theorem}

Our last two results concern divisibility of the linking number.
Fleming and Diesl~\cite{fleming-diesl2005} showed that for $q=3$ or
$q$ a power of two, embeddings of sufficiently large complete graphs
in $\real^3$ necessarily contain 2-component links with linking number
a nonzero multiple of $q$, and Fleming~\cite{fleming2007} later
extended this to all $q\in\naturalnumber$. We now extend this further
to embeddings of complete $n$-complexes in \ambient, and by slightly
modifying Fleming's argument, reduce the number of vertices required
from exponentially many to only polynomially many.  We state and prove
two results in this direction: the first is for $q$ arbitrary, and the
second is for $q=p$ prime, where a simpler argument leads to a bound
with much slower growth.

\begin{theorem}
\label{modq.th}
Let $q$ be a positive integer. Then for $N$ sufficiently large 
every embedding of \complete{N}\ in $\real^{2n+1}$ contains a 
two-component link $R\cup S$ such that $\link{R}{S}=kq$ for some
$k\neq 0$. The number of vertices required grows no faster than
$c(n+1)\binom{2n+4}{n+1}q^{n+2}$ ($c$ a constant), which for fixed $n$ grows
polynomially in $q$. 
\end{theorem}

When $q=p$ is prime, a much simpler argument leads to a bound with growth
$O(p^2)$ instead of $O(q^{n+2})$: 
\begin{theorem}
\label{modp.th}
Let $p\in\naturalnumber$ be prime, and let
\[
N = \key(2p-1)+n+\left\lceil\frac{2p-3}{n}\right\rceil+1.
\]
Then every embedding of \complete{N}\ in $\real^{2n+1}$ contains a 
two-component link $L\cup J$ such that $\link{L}{J}=kp$ for some
$k\neq 0$. 
\end{theorem}
Since the proof of Theorem~\ref{modp.th} is simpler than that of Theorem~\ref{modq.th} we will prove it first, in Section~\ref{modp.sec}, and then prove Theorem~\ref{modq.th} later in Section~\ref{modq.sec}.

For $n=1$, Theorem~\ref{modq.th} may be proved using a total of
\[
4q^2(6+15(q-1)) = 12q^2(5q-3)
\]
vertices, in contrast to the exponentially many required by
Fleming~\cite[Theorem 3.1]{fleming2007}. This reduction to polynomial
growth comes about for two reasons. The first is that we use
Flapan\etal's rather than Fleming and Diesl's construction of a
generalised keyring, as this requires only polynomially many rather
than exponentially many vertices. The second savings comes from
modifying the method by which the keys of the keyring are combined, so
that each key requires roughly $3q$ vertices rather than $O(q^{\log
  q})$. In fact it should be possible to reduce the number of vertices
required further, by a factor of about $2/3$, because for $n=1$ our
method really only requires the keys to have about $2q$ vertices.

For large $n$ 
Stirling's formula may be used to show that asymptotically we have
\[
c(n+1)\binom{2n+4}{n+1}q^{n+2}\sim C\sqrt{n}4^n q^{n+2}.
\]
Thus the number of vertices required grows at most exponentially with
respect to $n$.  

\subsection{Discussion}

We briefly discuss the existence of more complex links in embeddings
of large complete complexes in \ambient.

\subsubsection{More complex keyrings}

Each of Theorems~\ref{linkingnumber.th}--\ref{modp.th} is proved by
converting a suitable generalised keyring $R\cup L_1\cup\cdots\cup
L_m$ into a two component link $R\cup L'$, where $L'$ is formed as a
connect sum of some of the $L_i$ (and perhaps an additional disjoint
component $S$).  Starting with a generalised keyring with $mr$ keys,
and working with them $m$ at a time, we may therefore construct a link
$R\cup L'_1\cup\cdots\cup L'_r$ in which each linking number
\link{R}{L'_i}\ satisfies the conclusion of the theorem. It follows
for example that for $q\in\naturalnumber$ and $N$ sufficiently large,
every embedding of \complete{N}\ in \ambient\ contains a link $R\cup
L'_1\cup\cdots\cup L'_r$ in which each linking number
\link{R}{L'_i}\ is a nonzero multiple of $q$.

\subsubsection{More complex linking patterns}

Flapan\etal~\cite[Theorem 1]{flapan-mellor-naimi2008} show that
intrinsic linking of graphs in $\real^3$ is arbitrarily complex in the
following sense: Given natural numbers $r$ and $\lambda$, for $N$
sufficiently large every embedding of $K_N$ in $\real^3$ contains an
$r$-component link in which all pairwise linking numbers are at least
$\lambda$ in absolute value.  We believe that, with minor adaptions to
higher dimensions, their work shows that intrinsic linking of
$n$-complexes in \ambient\ is arbitrarily complex in this sense
also. The main adaption needed is to use our
Lemma~\ref{connectingsphere.lem} in place of the 1-dimensional
construction it replaces in higher dimensional arguments. This
adaption requires the addition of some extra vertices (to create the
auxiliary sphere $S_0$ of the lemma), and is illustrated in the proofs
of Lemma~\ref{makekeyring.lem} and
Theorem~\ref{linkingnumber.th}. These are based respectively on their
Lemma~1 and a technique from the proof of their Lemma~2.

A step in their argument is to show that, for $N$ sufficiently large,
every embedding of $K_N$ in $\real^3$ contains a link
$X_1\cup\cdots\cup X_m\cup Z_1\cup\cdots\cup Z_m$ such that
\[
\linktwo{X_i}{Z_j}=1
\]
for $1\leq i,j\leq m$
(Flapan\etal~\cite[Prop.\ 1]{flapan-mellor-naimi2008}).  We observe
that this step certainly extends to embeddings of complete
$n$-complexes in \ambient, as their proof is a purely
combinatorial argument that depends only on their Lemma~1 and the
existence of generalised keyrings, which we extend here to higher
dimensions as Lemma~\ref{makekeyring.lem} and Theorem~\ref{keyring.th}
respectively.

\subsection{Organisation}

The paper is organised as follows. We begin with some technical 
preliminaries in Section~\ref{spheres.sec}, and then prove 
Theorems~\ref{chains.th} and~\ref{keyring.th} concerning many-component
links in Section~\ref{rlinks.sec}. In Section~\ref{2-component.sec}
we prove our first two results on linking numbers in 
2-component links, Theorems~\ref{linkingnumber.th} and~\ref{modp.th}.

We then construct some triangulations of an $M$-simplex in
Section~\ref{triangulations.sec}, as further technical preliminaries
needed for our proof of our divisibility result
Theorem~\ref{modq.th}. This result is proved in
Section~\ref{modq.sec}. As a further application of the triangulations
of Section~\ref{triangulations.sec} we conclude the paper in
Section~\ref{linking-alternate.sec} with an alternate proof of
Theorem~\ref{linkingnumber.th}, without the polynomial bound on the
number of vertices required. This introduces an additional technique
that may be used to prove Ramsey-type results on intrinsic linking
of $n$-complexes. 

\section{Technical preliminaries I: Spheres and discs in \complete{N}}
\label{spheres.sec}

In this section we construct some subcomplexes of \complete{N}\
that are needed for our proofs. As an aid to understanding,
in Section~\ref{tactics.sec} we
first illustrate the role the corresponding subcomplexes of $K_N$
play in studying intrinsic linking of graphs in \rthree.

\begin{figure}
\begin{center}
\psfrag{a}{(a)}
\psfrag{b}{(b)}
\psfrag{c}{(c)}
\psfrag{d}{(d)}
\psfrag{X}{$X$}
\psfrag{X1}{$X_1$}
\psfrag{X2}{$X_2$}
\psfrag{Y1}{$Y_1$}
\psfrag{Y2}{$Y_2$}
\psfrag{v2}{$v_2$}
\psfrag{v1}{$v_1$}
\psfrag{w2}{$w_2$}
\psfrag{w1}{$w_1$}
\includegraphics[width=\textwidth]{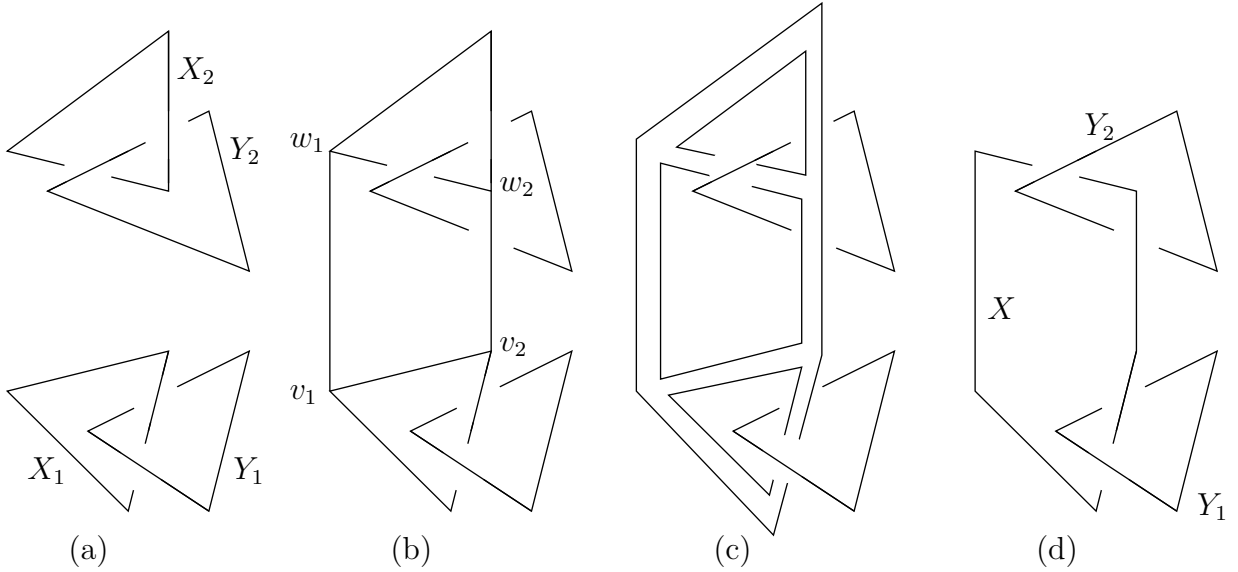}
\caption{Illustrating the proof of Lemma~\ref{4to3mod2.lem} (the
  four-to-three lemma for mod two linking number) in the case $n=1$.}
\label{linkinglemma.fig}
\end{center}
\end{figure}

\subsection{Tactics}
\label{tactics.sec}

A common technique of~\citeramsey\ in proving Ramsey-type results for
graphs is the use of connect sums and the additivity of linking
number.  These may be used to convert a link with several components
to one with fewer components, but more complicated linking
behaviour. We illustrate this technique by sketching the proofs for
$n=1$ of the four-to-three Lemmas~\ref{4to3mod2.lem}
and~\ref{4to3.lem}. The $n=1$ case of Lemma~\ref{4to3.lem} corresponds
to Lemma~2 of Flapan~\cite{flapan2002}, and Lemma~\ref{4to3mod2.lem}
is a mod two version of this result that is similar to Lemma~1 of
Flapan\etal~\cite{flapan-pommersheim-foisy-naimi2001}.

Suppose that the 4-component link $Y_1\cup X_1\cup X_2\cup Y_2$ in
Figure~\ref{linkinglemma.fig}(a) is part of an embedding of $K_N$ in
\rthree, and that we wish to replace the cycles $X_1$ and $X_2$ with a
single cycle $X$ linking both $Y_1$ and $Y_2$ mod two. We choose vertices
$v_1,v_2$ on $X_1$ and $w_1,w_2$ on $X_2$, and consider the edges
$(v_i,w_i)$ as in Figure~\ref{linkinglemma.fig}(b). Together with
$X_1$ and $X_2$ these give us a collection of cycles
(Figure~\ref{linkinglemma.fig}(c)) whose linking numbers with each of
$Y_1$ and $Y_2$ sum to zero mod two; and taking the connect sum of a
suitably chosen subset as in Figure~\ref{linkinglemma.fig}(d) we get the
desired cycle $X$.

Working now with integer co-efficients, consider the link $Y_1\cup
X_1\cup X_2\cup Y_2$ in Figure~\ref{integrallinking.fig}(a).  Our goal
here is to replace this with a three component link $L\cup Z\cup W$
such that \link{L}{Z}\ is nonzero, and \link{L}{W} is at least as
large as \link{X_2}{Y_2} in absolute value. We again do this by
constructing a series of cycles that sum to zero with $X_1$ and $X_2$,
but now in order to ensure we can find one linking $Y_2$ with the 
correct sign it is necessary to have at least 
$q>|\link{X_2}{Y_2}|$ such cycles. This is achieved by 
choosing vertices $v_1,\ldots,v_q$ on $X_1$ and 
$w_1,\ldots,w_q$ on $X_2$, such $v_1,\ldots,v_q$ are encountered
in increasing order following the orientation of $X_1$, and
$w_1,\ldots,w_q$ are encountered
in decreasing order following the orientation of $X_2$.
The needed cycles are formed by connecting 
$X_1$ and $X_2$ using the edges $(v_1,w_1),\ldots,(v_q,w_q)$,
as in Figure~\ref{integrallinking.fig}(b),
and a suitable connect sum 
(Figure~\ref{integrallinking.fig}(c)) then
gives us the desired 3-component link. 

\begin{figure}
\begin{center}
\psfrag{a}{(a)}
\psfrag{b}{(b)}
\psfrag{c}{(c)}
\psfrag{X}{$X$}
\psfrag{X1}{$X_1$}
\psfrag{X2}{$X_2$}
\psfrag{Y1}{$Y_1$}
\psfrag{Y2}{$Y_2$}
\psfrag{v0}{$v_1$}
\psfrag{v1}{$v_2$}
\psfrag{v2}{$v_3$}
\psfrag{v3}{$v_4$}
\psfrag{w0}{$w_1$}
\psfrag{w1}{$w_2$}
\psfrag{w2}{$w_3$}
\psfrag{w3}{$w_4$}
\includegraphics[width=\textwidth]{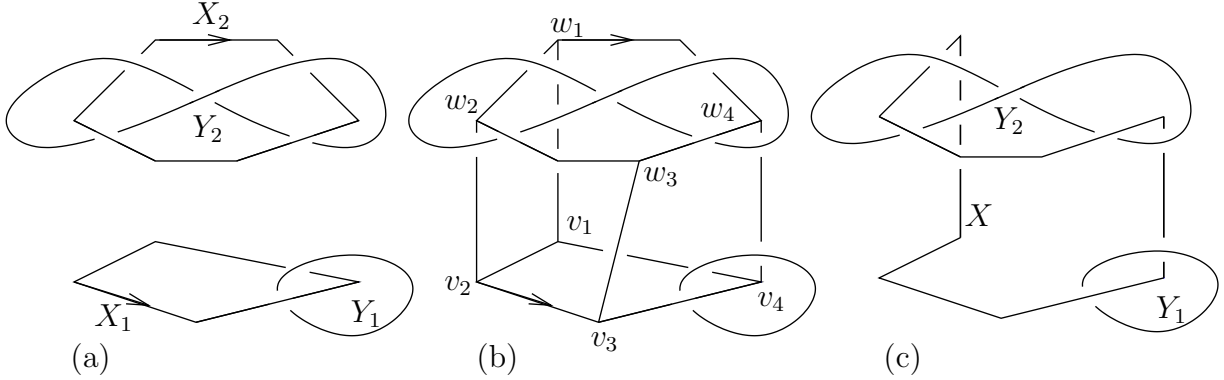}
\caption{Illustrating the proof of Flapan's Lemma~2, the $n=1$ case of
  our Lemma~\ref{4to3.lem} (the four-to-three lemma for integral
  linking number).}
\label{integrallinking.fig}
\end{center}
\end{figure}

To prove analogous results in higher dimensional dimensions we will
regard the intervals $[v_1,v_q]$ and $[w_1,w_q]$ as identically
triangulated discs $D_1\subseteq X_1$ and $D_2\subseteq X_2$, and the
correspondence $v_i\mapsto w_i$ as an orientation reversing
simplicial isomorphism $\phi:D_1\to D_2$ mapping one triangulation to the
other. Given this data we then construct the collection of edges
$(v_i,w_i)$, which we regard as a complex $\cyl$ homeomorphic to
$D_1^{(0)}\times I$ realising the restriction of $\phi$ to the zero
skeleton of $D_1$. The pair of edges $(v_i,w_i)$ and
$(v_{i+1},w_{i+1})$ may then be seen as a copy of $S^0\times I$, which
we cap with the intervals $[v_i,v_{i+1}],[w_i,w_{i+1}]$ to create
a copy of $S^1$. 

Triangulations of an interval have very simple combinatorics, and in
Figure~\ref{integrallinking.fig}(b) it didn't matter that there was
an additional vertex between $w_2$ and $w_3$. Thus, Flapan's argument
only requires that each component has at least $q$ vertices. In order
to use similar techniques when $n\geq 2$ we will impose the
more stringent requirement that our link components contain identically
triangulated copies of $D^n$. Additional work will then be required
to ensure that our links contain such discs.

\subsection{Cylinders, spheres and discs in \complete{N}}
\label{discs.sec}

We now construct the needed subcomplexes of \complete{N}.

\begin{lemma}
\label{crossed.lem}
Let $(S_1,D_1)$ and $(S_2,D_2)$ be disjoint subcomplexes of 
\complete{N}\ each homeomorphic to $(S^n,D^n)$. Suppose that there
is a simplicial isomorphism
\[
\phi:D_1\to D_2.
\]
Let $D_i^{(n-1)}$ be the $(n-1)$-skeleton of $D_i$. Then there is
a subcomplex $\cyl$ of \complete{N}\ and a homeomorphism
\[
\Phi : D_1^{(n-1)}\times I\rightarrow \cyl
\]
such that 
\begin{enumerate}
\item
all vertices of $\cyl$ lie on $D_1\cup D_2$;
\item
$\cyl\cap S_i= D_i^{(n-1)}$ for $i=1,2$;
\item
$\Phi$ restricts to the identity on $D_1^{(n-1)}\times\{0\}$; and
\item
$\Phi=\phi$ on $D_1^{(n-1)}\times\{1\}$.
\end{enumerate}
\end{lemma}

We note that the subcomplex $\cyl$ may be regarded as the mapping
cylinder of the restriction of $\phi$ to the $(n-1)$-skeleton.

\begin{proof}
To construct $\cyl$ we use the subdivision of $\Delta^k\times
I$ into $(k+1)$-simplices used in the proof of the homotopy invariance
of singular homology (see for example
Hatcher~\cite[p.\ 112]{hatcher-at}). Label the vertices of $D_1$
arbitrarily as $v_0,v_1,\ldots,v_M$, and label the vertices of $D_2$
as $w_0,w_1,\ldots,w_M$ so that $w_i=\phi(v_i)$.  Now, for each
$k$-simplex $\delta = [v_{i_0},\ldots,v_{i_k}]$ of $D_1^{(n-1)}$, with
$i_0<i_1<\cdots<i_k$, we have
\[
\delta\times I\cong \crossed{\delta}=
     \bigcup_{j=0}^k [v_{i_0},\ldots,v_{i_j},w_{i_j},\ldots,w_{i_k}].
\]
Since $k\leq n-1$ each $(k+1)$ simplex involved in this union is a
simplex of $\complete{N}$, and we obtain a subcomplex of
\complete{N}\ homeomorphic to $\delta\times I$, meeting $D_1$ and
$D_2$ in $\delta\times\{0\}=\delta$ and
$\delta\times\{1\}=\phi(\delta)$ respectively.
In addition, all vertices of $C(\delta)$ belong to
$D_1\cup D_2$. 

Let $\delta_l$ denote the simplex 
$[v_{i_0},\ldots,\hat{v}_{i_l},\ldots,v_{i_k}]$ belonging to $\partial\delta$,
where the hat indicates that $v_{i_l}$ is omitted. 
A $k$-simplex belonging to $\crossed\delta$ is of one of several possible
types:
\begin{enumerate}
\item
the simplex $[v_{i_0},\ldots,v_{i_k}]=\delta$ or 
$[w_{i_0},\ldots,w_{i_k}]=\phi(\delta)$;
\item
one of the simplices
\[
[v_{i_0},\ldots,\hat{v}_{i_l},\ldots,v_{i_j},w_{i_j},\ldots,w_{i_k}]
\]
or 
\[
[v_{i_0},\ldots,v_{i_j},w_{i_j},\ldots,\hat{w}_{i_l},\ldots,w_{i_k}]
\]
with $l$ fixed and $j\neq l$, which together make up $\crossed{\delta_l}$;
\item
a simplex of the form $[v_{i_0},\ldots,v_{i_j},w_{i_{j+1}},\ldots,w_{i_k}]$,
which is interior to $\delta\times I$. 
\end{enumerate}
Inductively, this implies that if $\delta'$ is a simplex of $\delta$,
then $\crossed{\delta'}$ is a subcomplex of \crossed\delta, and the
diagram
\[
\begin{CD}
\delta'\times I @>>>   \delta\times I \\
@V{\cong}VV            @VV{\cong}V \\
\crossed{\delta'} @>>> \crossed\delta
\end{CD}
\]
commutes. Moreover, our construction ensures that $\crossed{\delta_1}$
and $\crossed{\delta_2}$ are disjoint unless $\delta_1$ and $\delta_2$
intersect, in which case $\crossed{\delta_1}\cap\crossed{\delta_2} =
\crossed{\delta_1\cap\delta_2}$.  Thus, taking the union of
\crossed\delta\ over all $(n-1)$-simplices of $D_1^{(n-1)}$ we obtain
a subcomplex $\cyl$ of \complete{N}\ homeomorphic to $D_1^{(n-1)}\times I$ 
meeting $S_i$ in $D_i^{(n-1)}$ for each $i$, and the homeomorphism
$\Phi$ may be constructed satisfying the given conditions.
\end{proof}

\begin{corollary}
\label{embeddedspheres.cor}
Let $(S_1,D_1)$ and $(S_2,D_2)$ be disjoint subcomplexes of 
\complete{N}\ each homeomorphic to $(S^n,D^n)$. Suppose that there
is an orientation reversing simplicial isomorphism
\[
\phi:D_1\to D_2,
\]
and let $\Delta_1,\ldots,\Delta_k$ be the $n$-simplices of $D_1$. 
Then there are subcomplexes $P_0,P_1,\ldots,P_k$ of \complete{N}\ such
that
\begin{enumerate}
\item
\label{vertices.item}
the  vertices of $P_0,P_1,\ldots,P_k$ all lie on $S_1\cup S_2$;
\item
\label{spheres.item}
$P_i\cong S^n$ for each $i$;
\item
$P_0\cap S_j = \overline{S_j\setminus D_j}$ for $i=1,2$;
\item
\label{ends.item}
$P_i\cap S_1 = \Delta_i$, $P_i\cap S_2 = \phi(\Delta_i)$ for $i\geq 1$; and
\item
\label{zerosum.item}
as an integral chain we have
\[
S_1+S_2+\sum_{i=0}^k P_i = 0.
\]
\end{enumerate}
\end{corollary}

\begin{remark}
\label{disjoint.rem}
Condition~\eqref{vertices.item} implies that if $A$ is a subcomplex of
\complete{N}\ disjoint from $S_1\cup S_2$, then $A$ is disjoint
from $P_i$ for all $i$.
\end{remark}

\begin{proof}
We obtain the required spheres $P_i$ using the subcomplex $\cyl$ and
homeomorphism $\Phi:D_1^{(n-1)}\times I\to\cyl$ constructed in
Lemma~\ref{crossed.lem} above.  For each $i=1,\ldots,k$ let
\[
P_i = \Delta_i\cup\phi(\Delta_i)\cup\Phi(\partial\Delta_i\times I),
\]
and let
\[
P_0 = \overline{S_1\setminus D_1}\cup \overline{S_2\setminus D_2}
\cup \Phi(\partial D_1\times I).
\]
Then Lemma~\ref{crossed.lem} ensures that each $P_i$ is a
subcomplex of \complete{N}\ satisfying 
conditions~\eqref{vertices.item}--\eqref{ends.item} above.

To obtain~\eqref{zerosum.item} we must orient each sphere $P_i$.
For $i\geq 1$ we orient $P_i$ so that $\Delta_i$ receives the opposite
orientation from $P_i$ as it does from $S_1$, and we orient $P_0$ 
analogously using the disc $\overline{S_1\setminus D_1}$. This ensures
that $\phi_\sharp\Delta_i$ receives opposite orientations from $S_2$ and
$P_i$ also, since $\phi$ is orientation reversing on $\Delta_i$ with
respect to both $S_2$ and $P_i$ (on $P_i\cong S^n$ it is induced
by reflection in an equatorial $S^{n-1}$). Similar considerations
apply to $P_0$, as $\phi$ extends to a (not necessarily simplicial)
orientation reversing homeomorphism 
$(S_1,\overline{S_1\setminus D_1})\to(S_2,\overline{S_2\setminus D_2})$.

It remains to consider the subcomplexes $\crossed{\delta}$, for
$\delta$ an $(n-1)$-simplex of $D_1$. Each such simplex belongs to two
$n$-simplices of $S_1$, and receives opposite orientations from each
(since $\partial S_1=0)$;
consequently, each subcomplex $\crossed{\delta}$ belongs to two 
spheres $P_i$ and $P_j$, and is also oppositely oriented by each. This
completes the proof.
\end{proof}

\begin{remark}
The $n$-spheres $P_i$ of Corollary~\ref{embeddedspheres.cor}
may be expressed explicitly as chains
as follows. We assume throughout that all simplices of $D_1$ are written
with the labels on their vertices in increasing order. 

For each $k$-simplex $\delta = [v_{i_0},\ldots,v_{i_k}]$ of
$D_1^{(n-1)}$ define
\[
\prism(\delta) = \sum_{j=0}^k (-1)^j 
            [v_{i_0},\ldots,v_{i_j},w_{i_j},\ldots,w_{i_k}].
\]
Let $\eps_i\in\{\pm1\}$ be the co-efficient of 
$\Delta_i$ in the chain $S_1$, and set
\[
P_i = -\eps_i(\Delta_i+\prism(\partial\Delta_i)-\phi_\sharp(\Delta_i))
\]
for $i\leq 1$, and 
\[
P_0 = (D_1-S_1)+(D_2-S_2)+\prism\partial D_1. 
\]
We verify below that $\partial P_i = 0$, and that $S_1+S_2+\sum_i P_i = 0$. 

Suitably adapted, the calculation
on page~112 of Hatcher~\cite{hatcher-at} shows that 
\[
\partial\prism = \phi_\sharp-\id_\sharp - \prism\partial,
\]
so for $i\geq 1$ we have
\begin{align*}
-\eps_i\partial P_i &= \partial\Delta_i +\partial\prism\partial\Delta_i 
                             -\partial\phi_\sharp\Delta_i    \\
  &= \partial\Delta_i + \phi_\sharp\partial\Delta_i
        -\id_\sharp\partial\Delta_i - \prism\partial^2\Delta_i 
                             -\phi_\sharp\partial\Delta_i   \\
  &= 0.
\end{align*}
Similarly
\begin{align*}
\partial P_0 &= \partial(D_1 -S_1)+\partial(D_2-S_2) 
                 + \partial\prism\partial D_1 \\
             &= \partial D_1+\partial D_2
                   +\phi_\sharp\partial D_1 -\id_\sharp\partial D_1
                                    -\prism\partial^2D_1 \\
             &= \partial D_1 + \partial D_2 -\partial D_2 - \partial D_1
                & (\phi_\sharp\partial D_1 = \partial\phi_\sharp D_1
                                           = -\partial D_2) \\
             &=0,
\end{align*}
as required. Summing, we have $D_1 = \sum_{i=1}^k\eps_i\Delta_i$, so
\begin{align*}
\sum_{i=1}^k P_k &= -\sum_{i=1}^k\eps_i\Delta_i 
                    -\prism\partial\sum_{i=1}^k\eps_i\Delta_i
                        + \phi_\sharp\sum_{i=1}^k\eps_i\Delta_i \\ 
                 &= -D_1 - \prism\partial D_1 +\phi_\sharp D_1 \\
                 &= -D_1 - \prism\partial D_1 - D_2 \\
                 &= -P_0-S_1-S_2,
\end{align*}
and it follows that $S_1+S_2+\sum_{i=0}^k P_i = 0$.
\end{remark}

\subsection{Connect sums of several spheres}

Our next technical lemma takes several spheres $S_1,\ldots,S_k$, and
an additional sphere $S_0$, and constructs a sphere \sph\ meeting each
of $S_1,\ldots,S_k$ in a single $n$-simplex. The case
$n=1$, $k=3$ is illustrated
in Figure~\ref{connectingsphere.fig}. This
lemma is an adaption to higher dimensions of a construction used
by Flapan\etal~\cite{flapan-mellor-naimi2008} in 
the case $n=1$. In that case the additional sphere
$S_0$ is not needed, as it is only necessary to choose edges joining
$S_i$ to $S_{i+1}$, and $S_k$ to $S_1$. This depends on the fact that
the cylinder $S^0\times I$ is disconnected, and our additional sphere $S_0$
is necessary for $n\geq 2$, when $S^{n-1}\times I$ is connected.

\begin{figure}
\begin{center}
\psfrag{S0}{$S_0$}
\psfrag{S1}{$S_1$}
\psfrag{S2}{$S_2$}
\psfrag{S3}{$S_3$}
\psfrag{S}{$\sph$}
\psfrag{d1}{$\delta_1$}
\psfrag{d2}{$\delta_2$}
\psfrag{d3}{$\delta_3$}
\includegraphics[scale=0.6]{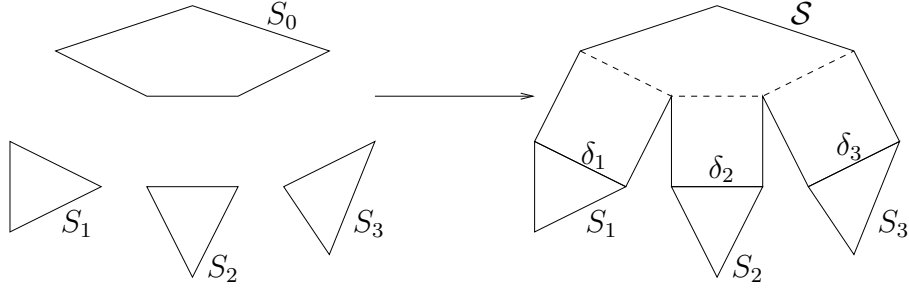}
\caption{Illustrating Lemma~\ref{connectingsphere.lem} in the
case $n=1$, $k=3$. The sphere
$S_0$ is used to construct a sphere \sph\ meeting $S_i$ in a single
$n$-simplex $\delta_i$ for $i=1,2,3$.}
\label{connectingsphere.fig}
\end{center}
\end{figure}

\begin{lemma}
\label{connectingsphere.lem}
Let $S_0,S_1,\ldots,S_k$ be disjoint subcomplexes of \complete{N}\
each homeomorphic to $S^n$, and suppose that $S_0$ has at least $k$
$n$-simplices. Then there is a subcomplex $\sph$ of 
\complete{N}\ such that
\begin{enumerate}
\item
the vertices of $\sph$ all lie on $S_0\cup\cdots\cup S_k$;
\item
$\sph$ is homeomorphic to $S^n$;
\item
for $i=1,\ldots,k$ there is an $n$-simplex $\delta_i$ of $S_i$ such that
$\sph\cap S_i=\delta_i$.
\end{enumerate}
Moreover, if each sphere $S_i$ is oriented, then $\sph$ may be chosen and
oriented such that $\delta_i$ receives opposite orientations from 
\sph\ and from $S_i$. 
\end{lemma}

\begin{proof}
We will assume that the $S_i$ are oriented. 
Choose an $n$-simplex $\delta_i$ belonging to $S_i$ for each $i\geq 1$,
distinct $n$-simplices $\delta'_i$ belonging to $S_0$ for
$i=1,\ldots,k$, and orientation reversing simplicial isomorphisms
$\phi_i:\delta_i\to\delta_i'$.  
Applying Corollary~\ref{embeddedspheres.cor} to the
pairs $(S_i,\delta_i)$ and $(S_0,\delta'_i)$ we obtain a sphere $Q_i$
with all its vertices on $S_i\cup S_0$, and such that $Q_i$ meets $S_i$ in
$\delta_i$ and $S_0$ in $\delta_i'$. 
Note that this implies $Q_i\cap Q_j=\delta_i'\cap\delta_j'$. 

We set $T_0=S_0$, and for $i=1,\ldots,k$ we inductively define $T_i$ to
be the complex obtained from $T_{i-1}$ and $Q_i$ by omitting the interior
of the disc $\delta_i'$. Then at each stage $T_i$ is 
an $n$-sphere, because it is the result of
gluing two discs along their common boundary $\partial\delta_i'$, and
setting $\sph=T_k$ we obtain the desired subcomplex.
\end{proof}

To conclude this section we establish a bound on the number of vertices
required to construct an $n$-sphere with a specified number of $n$-simplices.

\begin{lemma}
\label{n-sphere.lem}
Given $\ell\in\naturalnumber$ there is a triangulation of $S^n$ with
$n+\ell+1$ vertices and $\ell n+2$ $n$-simplices.
\end{lemma}

\begin{proof}
We construct the triangulation from a suitable triangulation of $D^{n+1}$
with $\ell$ $(n+1)$-simplices. For $i=1,\ldots,\ell$ let $\Delta_i$ be
an $(n+1)$-simplex, and choose distinct $n$-simplices $\delta_i$,
$\sigma_i$ belonging to $\Delta_i$. Choose a  
simplicial isomorphism $\phi_i:\delta_i\to\sigma_{i+1}$ for each
$i=1,\ldots,\ell-1$, and 
let $D$ be the $(n+1)$-disc that results from gluing the $\Delta_i$
according to the $\phi_i$. We claim that $S=\partial D$ is the
required triangulated $n$-sphere.

The union $\Delta_1\cup\cdots\cup\Delta_\ell$ has a total of
$\ell(n+2)$ $n$-simplices, of which $2(\ell-1)$ are identified in
pairs to form $D$. The $n$-simplices involved in the identifications
lie in the interior of $D$, and the rest on the boundary, so $S$ has
$\ell(n+2)-2(\ell-1)=\ell n+2$ $n$-simplices, as claimed. 
Similarly, each gluing identifies $2(n+1)$ vertices in pairs, leaving
a total of $\ell(n+2)-(n+1)(\ell-1)=\ell+n+1$; alternately, we may
carry the gluings out sequentially, and we see that we start with 
$n+2$ vertices, and each gluing adds just one, for a total of
$(n+2)+(\ell-1)=n+\ell+1$.

To complete the proof we show that the vertices of $D$ all lie on $S$.
For $n=1$ a circle with $\ell+2$ edges necessarily has $\ell+2$
vertices, by Euler
characteristic; while for $n\geq 2$ each vertex of $\Delta_i$
belongs to at least three $n$-simplices, and so to at least one
$n$-simplex belonging to $\partial D$ after the identifications.
\end{proof}

\begin{corollary}
\label{n-sphere.cor}
If $k\in\naturalnumber$ and $N\geq n+\lceil k/n\rceil+1$ then
\complete{N}\ contains a subcomplex $S\cong S^n$ with at least $k+2$
$n$-simplices.
\end{corollary}

\begin{proof}
Set $\ell=\lceil k/n\rceil$. Then $\ell\in\naturalnumber$ and
$\ell\geq k/n$, so the construction of Lemma~\ref{n-sphere.lem} yields
an $n$-sphere $S$ in $\complete{N}$ with at least $k+2$ $n$-simplices.
\end{proof}

\section{Many-component links}
\label{rlinks.sec}

We now prove Theorems~\ref{chains.th} and~\ref{keyring.th}, thereby
showing that embeddings of sufficiently large complete complexes
necessarily contain non-split links with many components.

\subsection{Necklaces and chains}

In this section we establish Theorem~\ref{chains.th}. The key step is
the following lemma, which plays the role of Lemma~1 in
Flapan\etal~\cite{flapan-pommersheim-foisy-naimi2001}.

\begin{lemma}[The four-to-three lemma for mod two linking number]
\label{4to3mod2.lem}
Let $Y_1\cup X_1\cup X_2\cup Y_2$ be a $4$-component link contained in 
some embedding of \complete{N}\ in \ambient, satisfying 
\[
\linktwo{X_1}{Y_1} = \linktwo{X_2}{Y_2} = 1.
\]
Then there is an $n$-sphere $X$ in \complete{N}, all of whose
vertices lie on $X_1\cup X_2$, such that
\[
\linktwo{Y_1}{X} = \linktwo{X}{Y_2} = 1.
\]
\end{lemma}

\begin{proof}
If $\linktwo{X_1}{Y_2}=1$ then we may simply let $X=X_1$, and if
$\linktwo{X_2}{Y_1}=1$ then we may simply let $X=X_2$. So suppose that
\[
\linktwo{X_1}{Y_2}=\linktwo{X_2}{Y_1}=0.
\]
Choose $n$-simplices $\delta_1$, $\delta_2$ belonging to $X_1$, $X_2$ 
respectively, and apply Corollary~\ref{embeddedspheres.cor} to 
the pairs $(X_1,\delta_1)$, $(X_2,\delta_2)$ to obtain spheres 
$P_0$, $P_1$ satisfying
\[
X_1 + X_2 + P_0 +P_1 = 0.
\]
In the homology groups $H_n(\ambient-Y_i;\Z{2})$ we have
\[
[X_1] + [X_2] + [P_0] +[P_1] = 0,
\]
and since $[X_1] + [X_2]=1$ in each group we have also
$[P_0] +[P_1] =  1$
in each group. Hence, for each $i$, precisely one of $[P_0]$, $[P_1]$
must equal 1 in $H_n(\ambient-Y_i;\Z{2})$. 

If $[P_1]$ takes the same value in both groups then we are done
by setting $X=P_0$ if $[P_1]=0$ in both groups, and $X=P_1$ if $[P_1]=1$.
Otherwise, without loss of generality suppose that $[P_1]$ is zero in
$H_n(\ambient-Y_1;\Z{2})$ and nonzero in $H_n(\ambient-Y_2;\Z{2})$,
and let $X$ be the $n$-sphere obtained from $X_1$ and $P_1$ by
omitting the interior of the simplex $\delta_1$. Then
\[
[X] = [X_1]+[P_1] =\begin{cases}
                   [X_1] = 1 & \text{in $H_n(\ambient-Y_1;\Z{2})$}, \\
                   [P_1] = 1 & \text{in $H_n(\ambient-Y_2;\Z{2})$},
                   \end{cases}
\]
and the result follows.
\end{proof}

We now prove Theorem~\ref{chains.th}, using the above lemma.

\begin{proof}[Proof of Theorem~\ref{chains.th}]
The proof of part~\eqref{chain.item} is by induction on $r$, with the
base case $r=2$ given by Taniyama~\cite{taniyama2000}, and the
inductive step following from Lemma~\ref{4to3mod2.lem}. Given an
embedding of \complete{(2n+4)r}\ in \ambient, choose
disjoint copies of \complete{(2n+4)(r-1)}\ and \complete{2n+4}\ contained
in the embedding. By the inductive hypothesis the
\complete{(2n+4)(r-1)}\ contains an $r$-component link $L_1\cup
L_2\cup\cdots\cup L_r$ satisfying
equation~\eqref{sequentiallinking.eq} for $i=1,\ldots,r-1$,
and the \complete{2n+4}\ contains a two component link $J\cup K$
such that $\linktwo{J}{K}=1$. Applying Lemma~\ref{4to3mod2.lem}
to the (ordered) link $L_{r-1}\cup L_r\cup J\cup K$ we obtain an $n$-sphere
$X$ with all its vertices on $L_r\cup J$ such that 
\[
\linktwo{L_{r-1}}{X}=\linktwo{X}{K}=1.
\]
The link $L_1\cup\cdots\cup L_{r-1}\cup X\cup K$ is then the desired
$r$-component link.

To prove~\eqref{necklace.item} we apply
Lemma~\ref{4to3mod2.lem} to suitably chosen components
of an $(r+1)$-component link as given by part~\eqref{chain.item}. Given
an embedding of \complete{(2n+4)r}\ in \ambient, there is an
$(r+1)$-component link $L_1\cup L_2\cup\cdots\cup L_r\cup L_{r+1}$
satisfying equation~\eqref{sequentiallinking.eq} for $i=1,\ldots,r$. 
We apply Lemma~\ref{4to3mod2.lem} to the
(ordered) link $L_r\cup L_{r+1}\cup L_1\cup L_2$ to obtain an $n$-sphere
$X$, with all its vertices on $L_{r+1}\cup L_1$, and satisfying
\[
\linktwo{L_{r}}{X}=\linktwo{X}{L_2}=1.
\]
The link $L_2\cup\cdots\cup L_{r}\cup X$ is then the desired
$r$-component link.
\end{proof}

\subsection{Generalised keyrings}

We prove Theorem~\ref{keyring.th}, by extending
Lemma~1 of Flapan\etal~\cite{flapan-mellor-naimi2008}
to higher dimensions in the following form.

\begin{lemma}
\label{makekeyring.lem}
Let \complete{N}\ be embedded in \ambient\ such that it contains
a link 
\[
L\cup J_1\cup\cdots\cup J_{m^2}\cup X_1\cup\cdots\cup X_{m^2},
\]
where $L$ has at least $m^2$ $n$-simplices, and $\linktwo{J_i}{X_i}=1$
for all $i$. Then there is an $n$-sphere $Z$ in \complete{N}\ 
with all its vertices on $L\cup J_1\cup\cdots\cup J_{m^2}$, and an
index set $I$ with $|I|\geq\frac{m}2$, such that 
$\linktwo{Z}{X_j}=1$ for all $j\in I$.
\end{lemma}

\begin{proof}
The argument is that of 
Flapan\etal~\cite{flapan-mellor-naimi2008}, with the addition of the
component $L$ needed to create the analogue of their cycle $C$
connecting the $J_i$. 

Since $L$ has at least $m^2$ simplices we may apply
Lemma~\ref{connectingsphere.lem} to the (ordered) link
$L\cup J_1\cup\cdots\cup J_{m^2}$, obtaining an $n$-sphere
\sph\ with all its vertices on $L\cup J_1\cup\cdots\cup J_{m^2}$
and meeting each sphere $J_i$ in an $n$-simplex $\delta_i$.  
If at least $\frac{m}2$ of the mod two linking
numbers $\linktwo{\sph}{X_i}$ are nonzero then we are done by setting
$Z=\sph$, so we assume in what follows that fewer than $\frac{m}2$ of
these mod two linking numbers are nonzero.

Following Flapan\etal\ we define $M$ to be the $m^2\times m^2$ matrix over
$\integer/2\integer$ with $ij$-entry $M_{ij}=\linktwo{J_i}{X_j}$. Let
$r_i$ be the $i$th row of $M$. Then $M_{ii}=1$ for all $i$, and
Flapan\etal\ use this to show that there are indices $i_1,\ldots,i_k$ such
that
\[
V=r_{i_1}+\cdots+r_{i_k}
\]
has at least $m$ entries that are equal to $1$. Let $Z$ be
the $n$-sphere obtained from $\sph$ and $J_{i_1},\ldots,J_{i_k}$ by omitting
the interiors of the simplices $\delta_{i_1},\ldots,\delta_{i_k}$.
We claim that $Z$ is the required $n$-sphere.

Indeed, for $j=1,\ldots,m^2$ we have
\begin{equation}
\label{rowsumlinking.eq}
\linktwo{Z}{X_j} = \linktwo{\sph}{X_j} + \sum_{\ell=1}^k \linktwo{J_{i_\ell}}{X_j}
                 = \linktwo{\sph}{X_j} + V_j,
\end{equation}
where $V_j = \sum_{\ell=1}^k \linktwo{J_{i_\ell}}{X_j}$ is the $j$th
entry of $V$. 
By construction at least $m$ of the $V_j$ are nonzero,
and by assumption fewer than $\frac{m}2$ of the \linktwo{\sph}{X_j}\ are
nonzero. Hence there are at least $m-\frac{m}{2}=\frac{m}{2}$ indices $j$ for
which $V_j=1$ while $\linktwo{\sph}{X_j}=0$. Consequently, the set
$I=\{1\leq j\leq m^2: \linktwo{\sph}{X_j} \neq V_j\}$ has at least $\frac{m}2$
elements. But $\linktwo{Z}{X_j}=1$ if and only if $j\in I$,
by~\eqref{rowsumlinking.eq}, so we are done.
\end{proof}

We now obtain Theorem~\ref{keyring.th} as a corollary to
Lemma~\ref{makekeyring.lem} and Corollary~\ref{n-sphere.cor}.

\begin{proof}[Proof of Theorem~\ref{keyring.th}]
Recall that 
\[
\key(r)=4r^2(2n+4)+n+\left\lceil\frac{4r^2-2}{n}\right\rceil+1,
\]
and for ease of notation 
let $\ell=\lceil{(4r^2-2)}/{n}\rceil$.
Given an embedding of \complete{\key(r)}\ in \ambient, choose $4r^2$
disjoint copies of \complete{2n+4}\ contained in the embedding,
together with a copy of
\complete{n+\ell+1}. By
Taniyama~\cite{taniyama2000} the $i$th copy of \complete{2n+4}\
contains a 2-component link $J_i\cup X_i$ such that $\linktwo{J_i}{X_i}=1$,
and by Corollary~\ref{n-sphere.cor} the copy of $\complete{n+\ell+1}$
contains an $n$-sphere $L$ with at least $4r^2$ $n$-simplices. 
The result now follows by applying
Lemma~\ref{makekeyring.lem} with $m=2r$ to the link
\[
L\cup J_1\cup\cdots\cup J_{4r^2}\cup X_1\cup\cdots\cup X_{4r^2}.
\qedhere
\]
\end{proof}

\section{Linking number in 2-component links}
\label{2-component.sec}

We now prove Theorems~\ref{linkingnumber.th} and~\ref{modp.th},
concerning the linking number in a 2-component link. To prove each result
we start with a suitable generalised keyring, and combine some of the
``keys'' to obtain the second component of the desired link.

\subsection{Bounding the absolute value of the linking number from below.}

\begin{proof}[Proof of Theorem~\ref{linkingnumber.th}]
We use a technique of Flapan\etal~\cite{flapan-mellor-naimi2008}
from the proof of their Lemma~2.  For simplicity
of notation let $\ell = \lceil(2\lambda-1)/n\rceil$, and choose disjoint
copies of \complete{\key({2\lambda-1})}\ and
\complete{n+\ell+1}\ contained in \complete{N}. Given an embedding of
\complete{N}\ in \ambient, the copy of
\complete{\key({2\lambda-1})}\ contains a generalised keyring 
$R\cup L_1\cup\cdots\cup L_{2\lambda-1}$ with
$2\lambda-1$ keys, by Theorem~\ref{keyring.th}, while by
Corollary~\ref{n-sphere.cor} the copy of \complete{n+\ell+1}\ contains an
$n$-sphere $S$ with at least $2\lambda+1$ $n$-simplices. 

Orient $S$ arbitrarily, and orient the $L_i$ such that
$\link{R}{L_i}>0$ for each $i$. Applying
Lemma~\ref{connectingsphere.lem} to the oriented link $\mathcal{L}=S\cup
L_1\cup\cdots\cup L_{2\lambda-1}$ we obtain an $n$-sphere \sph\ with
all its vertices on $\mathcal{L}$ and meeting each $L_i$ in a single
$n$-simplex $\delta_i$, which
receives opposite orientations from $\sph$ and from $L_i$. 
Set $\sph_0=\sph$, and for $i=1,\ldots,2\lambda-1$ let $\sph_i$ be
the complex obtained from $\sph_{i-1}$ and $L_i$ by omitting the
interior of the disc $\delta_i$. Then $\sph_i$ is an $n$-sphere, because
it is the result of gluing
two discs along their common boundary $\partial\delta_i$, and as
a chain we have
\begin{equation}
\label{sumofspheres.eq}
\sph_i = \sph_0 + \sum_{j=1}^i L_j
\end{equation}
for $i\geq 1$. 

We now consider the linking numbers of the $\sph_i$ with $R$,
by considering equation~\eqref{sumofspheres.eq} in the
group $H_n(\ambient-R;\integer)$. This gives
\[
\link{R}{\sph_i} = [\sph_i] = [\sph_0]+\sum_{j=1}^i [L_j]
                            = \link{R}{\sph_0}+\sum_{j=1}^i \link{R}{L_j}.
\]
As in the proof of Lemma~2 of Flapan\etal\ the sequence
$\bigl(\link{R}{\sph_i}\bigr)_{i=0}^{2\lambda-1}$ is strictly
increasing, because the linking numbers $\link{R}{L_i}$ are all
positive. This sequence must therefore take $2\lambda$ distinct
values, and the result now follows from the fact that there are only
$2\lambda-1$ integers $k$ such that $|k|<\lambda$.
\end{proof} 

\subsection{The linking number modulo a prime $p$}
\label{modp.sec}

To prove Theorem~\ref{modp.th} we will use the following lemma on
sums of subsequences of finite integer sequences, considered mod
$p$. Given an integer sequence $(\ell_1,\ldots,\ell_m)$ we will
say that $x\in\integer$ is a \emph{subsequence sum} of
$(\ell_1,\ldots,\ell_{m})$ if there is
a subset $A\subseteq\{1,\ldots,m\}$ such that
\[
\sum_{i\in A}\ell_i = x.
\]
We allow the possibility that $A$ is empty, which implies that
$0$ is always a subsequence sum. Then:

\begin{lemma}
\label{subsequencesums.lem}
Let $p\in\naturalnumber$ be prime, and let $(\ell_1,\ldots,\ell_{p-1})$
be a sequence of integers such that no $\ell_i$ is divisible by $p$.
For any $s\in\integer$ there is a subsequence sum $x$ of
$(\ell_1,\ldots,\ell_{p-1})$ such that $x\equiv s\bmod p$.
\end{lemma}

We note that the sequence length $p-1$ is best possible, because a
sequence of length $p-2$ that is constant mod $p$ realises 
exactly $p-1$ mod $p$ residue classes as subsequence sums.

\begin{proof}
For $j=1,\ldots,p-1$ let $\Sigma_j$ be the set of mod $p$ residue
classes that may be realised by a subsequence sum of
$(\ell_1,\ldots,\ell_{j})$. Then
$\Sigma_1=\bigl\{\overline0,\overline{\ell_1}\bigr\}$, and our goal is
to show that
$\Sigma_{p-1}=\bigl\{\overline0,\overline1,\ldots,\overline{p-1}\bigr\}$. We
will do this by showing that $|\Sigma_{j+1}|\geq|\Sigma_j|+1$ whenever
$\Sigma_{j}\neq\bigl\{\overline0,\overline1,\ldots,\overline{p-1}\bigr\}$. Since
$\Sigma_j\subseteq\Sigma_{j+1}$ it suffices to show that there is an
element of $\Sigma_{j+1}$ that is not an element of $\Sigma_j$.

Suppose then that $\Sigma_j\neq\bigl\{\overline0,\overline1,\ldots,\overline{p-1}\bigr\}$,
and consider multiples of $\ell_{j+1}$ mod $p$. Since 
$\ell_{j+1}\not\equiv0\bmod p$ we have
\[
\bigl\{\overline{k\ell_{j+1}}|0\leq k\leq p-1\bigr\}
= \bigl\{\overline0,\overline1,\ldots,\overline{p-1}\bigr\}
\supsetneq \Sigma_j\supseteq\bigl\{\overline0\bigr\},
\]
so there is some $1\leq k\leq p-1$ such that 
$\overline{k\ell_{j+1}}\notin\Sigma_j$. Consider the least such 
$k$. Then there is a (possibly empty, if $k=1$) subset $A$ of
$\{1,\ldots,j\}$ such that 
\[
\sum_{i\in A} \ell_i \equiv (k-1)\ell_{j+1}\bmod p,
\]
and setting $B=A\cup\{j+1\}$ we have
\[
\sum_{i\in B} \ell_i \equiv k\ell_{j+1}\bmod p.
\]
Hence $\overline{k\ell_{j+1}}$ belongs to $\Sigma_{j+1}$ but not
$\Sigma_j$, and we are done.
\end{proof}

\begin{proof}[Proof of Theorem~\ref{modp.th}]
The technique is similar to that used in the proof of
Theorem~\ref{linkingnumber.th}.  By Theorem~\ref{keyring.th} and
Corollary~\ref{n-sphere.cor}, $N$ is so large that every embedding of
\complete{N}\ in \ambient\ contains a generalised keyring $R\cup
L_1\cup\cdots\cup L_{2p-1}$ with $2p-1$ keys, and an additional
disjoint sphere $S$ with at least $2p-1$ $n$-simplices.  Orient the
link $S\cup L_1\cup\cdots\cup L_{2p-1}$ as in the proof of
Theorem~\ref{linkingnumber.th}, and let \sph\ be the $n$-sphere that
results from applying Lemma~\ref{connectingsphere.lem} to this link.

We now consider the linking numbers $\link{R}{\sph}$ and
$\link{R}{L_i}$ modulo $p$. If $\link{R}{L_i}\equiv 0$ for
some $i$ then we are done, so we may assume that all such
linking numbers are nonzero mod $p$. Then by Lemma~\ref{subsequencesums.lem}
there is a subset $A\subseteq\{1,\ldots,p-1\}$ such that
\[
\sum_{i\in A} [L_i] \equiv -[\sph] \bmod p,
\]
and a subset $B\subseteq\{p+1,\ldots,2p-1\}$ such that
\[
\sum_{i\in B} [L_i] \equiv -[L_p] \bmod p.
\]
Set $C=B\cup\{p\}$, to obtain a nonempty subset of
$\{p,\ldots,2p-1\}$ such that 
\[
\sum_{i\in C} [L_i] \equiv 0 \bmod p.
\]
We now consider the chains
\begin{align*}
S_1 &= \sph + \sum_{i\in A} L_i,    &
S_2 &= S_1 + \sum_{i\in C} L_i.
\end{align*}
In the homology group
$H_n(\ambient-R;\integer)$ we have
\[
[S_1]\equiv[S_2]\equiv 0\bmod p,
\]
and moreover $[S_1]\neq[S_2]$, because the linking numbers $[L_i]$
are all positive and $C$ is nonempty. It follows that at least
one of $[S_1]$ and $[S_2]$ is nonzero, and since both chains represent
$n$-spheres we are done.
\end{proof}

We note that the argument used above does require $p$ to be prime. For
$q$ composite, if $\link{R}{\sph}$ is coprime to $q$ and all linking numbers
$\link{R}{L_i}$ are equal to the same nontrivial divisor $d$ of $q$, then
no sphere formed from \sph\ and the $L_i$ as above will link $R$ with linking
number divisible by $q$.
We will therefore use a different strategy in Section~\ref{modq.sec}
to prove the corresponding
result when $q$ may be composite.

\section{Technical preliminaries II: Triangulations of an $M$-simplex}
\label{triangulations.sec}

We now establish some additional technical preliminaries needed to
prove Theorem~\ref{modq.th}. For this theorem we will need to work
with links containing identically triangulated discs $D^n$ with many
$n$-simplices, and to this end we will construct a triangulation of an
$M$-simplex into many $M$-simplices.

\subsection{The triangulations}

For $\ell\in\naturalnumber$ let $\Delta^M_\ell$ be the $M$-simplex 
\[
\Delta = 
 [\ell\vect{e}_1,\ell\vect{e}_2,\ldots,\ell\vect{e}_{M+1}]\subseteq\real^{M+1},
\]
where $\vect{e}_1,\vect{e}_2,\ldots,\vect{e}_{M+1}$ are the standard
basis vectors. Then:

\begin{lemma}
\label{triangulations.lem}
The family of planes
\begin{equation}
\label{dividingplanes.eq}
\left\{\left.\sum_{k=i}^j x_k \in\integer\right| 
              1\leq i\leq j\leq M\right\}
\end{equation}
subdivides $\Delta^M_\ell$ into $\ell^M$ $M$-simplices. The symmetry
group of this triangulation is the dihedral group $D_{M+1}$ of order
$2(M+1)$, with the action given by permutations of the basis vectors
$\vect{e}_i$ that preserve or reverse the cyclic ordering 
$\vect{e}_1,\vect{e}_2,\ldots,\vect{e}_{M+1}$. 
\end{lemma}

We will call an $M$-simplex triangulated as in Lemma~\ref{triangulations.lem}
a \emph{triangulated $M$-simplex of side-length $\ell$}, and denote it
by $\divided[M]{\ell}$. 

\begin{remark}
The triangulation \divided[2]{\ell}\ is simply the standard division
of an equilateral triangle of side-length $\ell$ into $\ell^2$ equilateral
triangles of side-length $1$. In this case all simplices of the triangulation
are isometric. However, for $M\geq 3$ the simplices of the triangulation
may no longer all be isometric. This may be seen in the case \divided[3]{2},
where four of the $3$-simplices are regular tetrahedra, and the remaining
four are obtained by cutting an octahedron along two of the three planes
of symmetry that pass through four vertices.
\end{remark}

\begin{remark}
\label{orientation.rem}
The $M+1$ cycle $(1\;2\;\ldots\;M+1)$ in $D_{M+1}$ reverses orientation 
of $\Delta_l^M$ if and only if $M$ is odd, and when $M$ is even the order
two elements of $D_{M+1}$ reverse orientation if and only if 
$M\equiv2\bmod 4$. So \divided[M]{\ell}\ has an orientation reversing
symmetry if and only if $M\not\equiv 0\bmod 4$.
\end{remark}

\begin{proof}
We proceed by subdividing the simplex
\[
\Sigma^M_\ell = \{\vect{x}\in\real^M|
          0\leq x_1\leq x_2\leq \cdots \leq x_M\leq \ell\} 
\]
into $\ell^M$ simplices,
and then pull this subdivision back to $\Delta^M_l$. The chief reason
for working with $\Delta_\ell^M$ rather than $\Sigma_\ell^M$ is that the
symmetries of the triangulation are more readily seen.

We first observe that for each permutation $\sigma\in S_M$, the set
\[
\delta_\sigma = \{\vect{x}\in\real^M|
  0\leq x_{\sigma(1)}\leq x_{\sigma(2)}\leq \cdots \leq x_{\sigma(M)}\leq 1\}
\]
is an $M$-simplex, and that the collection of such simplices gives
a subdivision of $I^M$ into $M!$ simplices. These simplices are
defined by the family of planes
\[
\{x_i=0\}\cup\{x_i=1\}\cup\{x_j-x_i=0\},
\]
and translating these according to $\integer^M\leq\real^M$ we see that
the family 
\begin{equation}
\label{dividingRM.eq}
\{x_i\in\integer|1\leq i\leq M\}
\cup\{x_j-x_i\in\integer|1\leq i< j\leq M\}
\end{equation}
gives a subdivision of all of $\real^M$ into isometric simplices. The
planes bounding $\Sigma_\ell^M$ belong to this family, and it follows
that the subdivision of $\real^M$ restricts to a subdivision of
$\Sigma_\ell^M$. This subdivision must have $\ell^M$ simplices, on
purely volumetric grounds.

We now pull this triangulation back to $\Delta_\ell^M$ via the linear
map that sends the vertex $\vect{e}_i$ of $\Delta_\ell^M$ to the
vertex $\vect{e}_i+\cdots+\vect{e}_M$ of $\Sigma_\ell^M$ for $i\leq M$,
and the vertex $\vect{e}_{M+1}$ to the vertex $\vect{0}$. Let 
$\{\phi_i\}$ be the dual basis to $\{\vect{e}_i\}$. 
Then $\phi_i$ pulls back to $\phi_1+\cdots+\phi_i$, and we see that 
the family~\eqref{dividingRM.eq} pulls back
to the family~\eqref{dividingplanes.eq}. This
linear map induces an affine homeomorphism between $\Sigma_\ell^M$ and
$\Delta_\ell^M$, and so these planes give us the desired triangulation. 

To see that the symmetry group is $D_{M+1}$, we observe that on
the plane $\sum x_i = \ell$ containing $\Delta_\ell^M$, the 
conditions
\[
\sum_{k=i}^j x_k \in\integer \qquad \text{and} \qquad
\sum_{k=1}^{i-1} x_k + \sum_{k=j+1}^{M+1}x_k\in\integer
\]
are equivalent. Thus, each family of planes defining the subdivision
may be viewed as a division of a necklace of $M+1$ beads into two
connected components, and conversely. Symmetries of the triangulation
therefore correspond to precisely those permutations of the beads
that preserve adjacency, giving us $D_{M+1}$. 
\end{proof}

\begin{construction}
For $M\geq n+1$ we define \comdiv{M}{\ell}\ to be the subcomplex of
\divided[M-1]{\ell}\ consisting of precisely those simplices lying
entirely within the $n$-skeleton 
$(\Delta_\ell^{M-1})^{(n)}\cong \complete{M}$. 
Each $n$-simplex of $(\Delta_\ell^{M-1})^{(n)}$ lies in an 
$n$-dimensional co-ordinate plane, and is isometric to
$\Delta_\ell^n$; intersecting the family of planes~\eqref{dividingplanes.eq}
with this subspace subdivides this simplex into a \divided{\ell}. 
Thus \comdiv{M}{\ell}\ is a space homeomorphic to \complete{M}, with each
$n$-simplex of \complete{M}\ mapping onto a copy of \divided{\ell}.
As such we will call it a \emph{triangulated complete $n$-complex on $M$
vertices of side-length~$\ell$}.
\end{construction}

\subsection{Counting the vertices}

The number of vertices in a \divided[k]{\ell}\ is equal to the
number of non-negative integer solutions to the equation
\[
x_1+x_2+\cdots+x_{k+1} = \ell,
\]
and the number of vertices in the interior of a \divided[k]{\ell}\
is the number of positive integer solutions to this equation.
These numbers are $\binom{k+\ell}{k}$ and 
$\binom{\ell-1}{k}=\binom{\ell-1}{\ell-k-1}$
respectively. Counting the vertices of a \divided[M]{\ell}\ according 
to the open simplex of $\Delta_\ell^M$ that they belong to we
find that it has
\begin{equation}
\label{dividedMlvertexcount.eq}
\sum_{k=0}^{M} \binom{M+1}{k+1}\binom{\ell-1}{\ell-k-1} 
= \binom{\ell+M}{M}
\end{equation}
vertices (the two sides are the co-efficient of $x^\ell$ in 
$(1+x)^{M+1}(1+x)^{\ell-1} = (1+x)^{\ell+M}$). 

Of particular interest is the number of vertices belonging to 
$\comdiv{2n+4}{\ell}$, as this complex is homeomorphic to
\complete{2n+4}, and may be used to construct links in which
each component has many $n$-simplices.
Setting $M=2n+3$ in equation~\eqref{dividedMlvertexcount.eq},
and truncating the sum at $k=n$, we therefore find that
$\comdiv{2n+4}{\ell}$ has a total of
\begin{equation}
\label{vertexcount.eq}
V(n,\ell)=\sum_{k=0}^{n} \binom{2n+4}{k+1}\binom{\ell-1}{k}
\end{equation}
vertices. 

For a more tractable bound, observe that the triangulated simplex
$\divided{\ell}$ has $\ell^n$ $n$-simplices, each with $n+1$ vertices,
and so has at most $(n+1)\ell^n$ vertices. The complex
$\comdiv{2n+4}{\ell}$ contains $\binom{2n+4}{n+1}$ such triangulated
simplices, and therefore
\[
V(n,\ell) \leq (n+1)\binom{2n+4}{n+1}\ell^n
\]
(this also follows from the inequalities 
$\binom{2n+4}{k+1}\leq\binom{2n+4}{n+1}$ and
$\binom{\ell-1}{k}\leq\ell^n$ for $k\leq n$). Stirling's formula 
$m!\sim\sqrt{2\pi m}(m/e)^m$ leads to the asymptotic formula
$\binom{2m}{m}\sim4^m/\sqrt{\pi m}$, and hence
\[
(n+1)\binom{2n+4}{n+1} = \frac{(n+1)(n+2)}{n+3}\binom{2(n+2)}{n+2}
                       \sim \sqrt{\frac{n}\pi}\, 4^{n+2}
                       = C\sqrt{n} 4^n.
\]
Consequently, asymptotically $V(n,\ell)$ grows no faster than
$C\sqrt{n} (4\ell)^n$.

\section{Linking number mod $q$}
\label{modq.sec}

The goal of this section is to prove Theorem~\ref{modq.th}, which we
recall states that given $q\in\naturalnumber$, embeddings of
sufficiently large complete $n$-complexes in \ambient\ contain
2-component links with linking number a nonzero multiple of $q$.
Before proving this theorem we need one more technical lemma:

\begin{lemma}
\label{lineardisc.lem}
Let $R$ be a positive integer. For $\ell$ sufficiently large 
$\divided{\ell}$ contains a triangulated disc $D$ with
$r\geq R$ $n$-simplices $\Delta_1,\ldots,\Delta_r$, which may
be labelled such that 
\[
D_{ij} = \bigcup_{k=i}^j \Delta_k
\]
is a disc for any $1\leq i\leq j\leq r$. The conclusion holds
for $\ell\geq R$, so the side length required grows at most
linearly with $R$. 
\end{lemma}

\begin{proof}
Write \dividedsigma{\ell}\ for the $n$-simplex $\Sigma^n_\ell$
subdivided by the family of planes given by
equation~\eqref{dividingRM.eq}.  Then \dividedsigma{\ell}\ and
\divided{\ell}\ are simplicially isomorphic, so it suffices to
construct a suitable disc $D$ in \dividedsigma{\ell}. We will
construct $D$ as the union of the $n$-simplices of
\dividedsigma{\ell}\ that meet a suitably chosen line $L$ in $\real^n$.
The case $n=2$, $\ell=4$ is illustrated in Figure~\ref{lineardisc.fig}.

\begin{figure}
\begin{center}
\psfrag{x}{$x_1$}
\psfrag{y}{$x_2$}
\psfrag{L}{$L$}
\includegraphics[scale=0.7]{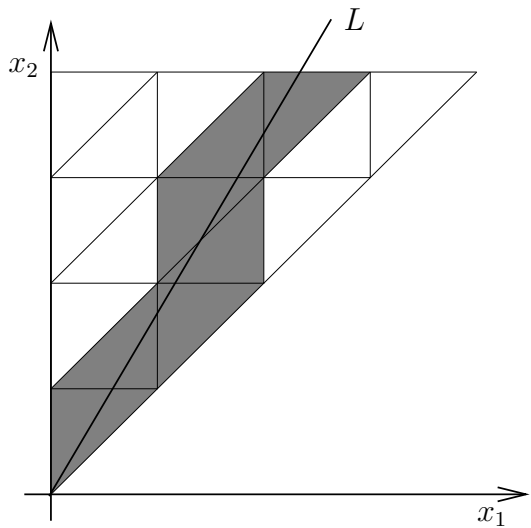}
\caption{Illustrating the construction of the disc $D$ of
  Lemma~\ref{lineardisc.lem} in the case $n=2$, $\ell=4$. A line $L$ with
  irrational slope $\alpha>1$ meets each line defining the
  triangulation exactly once, and except at \vect0\ never passes
  through the intersection of two such lines. We take $D$ to be the
  union of the 2-simplices intersecting $L$ (shaded grey). The disc
  $D$ contains at least $\ell$ $n$-simplices (here at least 4), since it must
  include at least one from each horizontal slice.}
\label{lineardisc.fig}
\end{center}
\end{figure}

Since $\real$ is infinite dimensional as a vector space over
$\rational$, we may choose $0<\alpha_1<\cdots<\alpha_n=1$ such that
$\{\alpha_1,\ldots,\alpha_n\}$ is linearly independent over 
\rational. Write $\vectg{\alpha}=(\alpha_1,\ldots,\alpha_n)$,
and let $L$ be the line $L=\{t\vectg\alpha:t\in\real\}$. 
Each plane in the family~\eqref{dividingRM.eq} may be written
in the form $\vect{c}^T\vect{x}=u$, where
$\vect{c}\in\integer^n$ and $u\in\integer$, and the linear independence
of $\{\alpha_1,\ldots,\alpha_n\}$ over \rational\ may be
used to show that
\begin{enumerate}
\item
$L$ meets each plane in the family~\eqref{dividingRM.eq} transversely; and
\item
each point of $L$ other than \vect0\ lies on at most one plane in
this family.
\end{enumerate}
Together these facts imply that, with the exception of simplices containing
\vect0, $L$ can meet only $n$- and $(n-1)$-simplices
of \dividedsigma\ell, and that if it intersects an $n$-simplex at
all it must intersect it in its interior. 

Observe that the line segment $\{t\vectg\alpha:0\leq t\leq\ell\}$ is
contained in $\Sigma^n_\ell$, and cuts each plane
$x_n=k$ for $k=1,\ldots,\ell$. Consequently $L$ must pass through at
least one $n$-simplex of \dividedsigma\ell\ lying in the slice
$\{\vect{x}:k-1\leq x_n\leq k\}$ for each $1\leq k\leq\ell$, and so 
passes through at least $\ell$ $n$-simplices of \dividedsigma\ell.
Suppose that $L$ passes through exactly $r$ $n$-simplices of 
\dividedsigma\ell, and label them consecutively
$\Delta_1,\ldots,\Delta_r$ in the order in which they are 
encountered when tracing $L$ in the direction \vectg\alpha.
We claim that
\[
D_{ij} = \bigcup_{k=i}^j \Delta_k
\]
is a disc for any $1\leq i\leq j\leq\ell$, from which the result follows.

Since the open ray $\{t\vectg\alpha:t>0\}$ only meets $n$- and
$(n-1)$-simplices of \dividedsigma\ell, consecutive $n$-simplices
$\Delta_k$ and $\Delta_{k+1}$ must intersect in an $(n-1)$-simplex. In
addition, for $d\geq2$ the simplices $\Delta_k$ and $\Delta_{k+d}$ are
separated by at least two planes from the
family~\eqref{dividingRM.eq}, and so meet in at most an
$(n-2)$-simplex. 
Since $D_{ii}=\Delta_i$ is a disc, and $D_{i,k+1}$ is the
result of gluing $D_{ik}$ and $\Delta_{k+1}$ along the
$(n-1)$-simplex $\Delta_k\cap\Delta_{k+1}$, it follows by induction
that $D_{ij}$ is a disc, as claimed.
\end{proof}

We now prove Theorem~\ref{modq.th}. The argument again proceeds by
converting a suitably large generalised keyring to a 2-component link,
but now we require additionally that the keys of the keyring are
copies of \comdiv{n+2}{q}. Our underlying approach is similar to that
of Fleming~\cite[Theorem 3.1]{fleming2007}, but differs from his in
the size of the keys and the method used to combine them to form the
second component of the link.

\begin{proof}[Proof of Theorem~\ref{modq.th}]
We show that the result holds for
\[
N=4q^2V(n,q)+n+\left\lceil\frac{4q^2-2}{n}\right\rceil+1,
\]
where $V(n,q)$ is given by~\eqref{vertexcount.eq} and equals the
number of vertices belonging to $\comdiv{2n+4}{q}$. Since
$V(n,q)\leq (n+1)\binom{2n+4}{n+1}q^n$, we conclude that $N$ grows no
faster than
$C(n+1)\binom{2n+4}{n+1}q^{n+2}$.

Given an embedding of \complete{N}\ in \ambient, let
$C_1,\ldots,C_{4q^2}$ be disjoint copies of
\comdiv{2n+4}{q}\ contained in \complete{N}, and use the remaining
$n+\lceil(4q^2-2)/n\rceil+1$ vertices and Corollary~\ref{n-sphere.cor}
to construct an $n$-sphere $L$ with at least $4q^2$
$n$-simplices. The complex $C_i$ is homeomorphic to \complete{2n+4}, and so
by Taniyama~\cite{taniyama2000} contains a two component link $J_i\cup
X_i$ such that $\link{J_i}{X_i}\neq 0$, and each component is a copy
of \comdiv{n+2}{q}. Applying Lemma~\ref{makekeyring.lem} to the link
$L\cup J_1\cup\cdots\cup J_{4q^2}\cup X_1\cup\cdots\cup X_{4q^2}$ we
obtain a generalised keyring $R\cup L_1\cup\cdots\cup L_q$, where
$\link{R}{L_i}\neq 0$ for each $i$, and each $L_i$ is a copy of
\comdiv{n+2}q.
We will use $R$ as one component of our link, and we will seek to
construct the second as a connect sum of some of the $L_i$. In what
follows we therefore consider homology classes in
$H_n(\ambient-R;\integer)$.

Orient the $L_i$ such that $\link{R}{L_i}=[L_i]$ is positive for each $i$,
and for $1\leq k\leq q$ consider the values of the sums
$\sum_{i=1}^k[L_i]$ mod $q$. Since there are $q$ sums and $q$ possible values 
modulo $q$, by the Pigeonhole Principle there must either be a sum
that is zero mod $q$, or else two sums that are equal modulo $q$. In
either case we obtain integers $a,b$ satisfying $1\leq a\leq b\leq q$
such that 
\[
\sum_{i=a}^b [L_i] \equiv 0\bmod q.
\]
From now on we restrict our attention to the spheres $L_a,\ldots,L_b$.

Our construction now departs from that of Fleming. Each component $L_i$ is
a copy of \comdiv{n+2}q, and as such has $n+2$ faces which are
triangulated $n$-simplices of sidelength $q$. We claim that it is
possible to choose distinct faces $\delta_i$, $\delta_i'$ of $L_i$,
each a copy of \divided{q}, and orientation reversing simplicial
isomorphisms $\psi_i:\delta_i\to\delta_{i+1}'$.  For $n\not\equiv
0\bmod 4$ this may be done by choosing distinct faces $\delta_i$,
$\delta_i'$ of $L_i$ arbitrarily, since in this case \divided{q}\ has
both orientation preserving and reversing symmetries, by
Remark~\ref{orientation.rem}. However, for $n\equiv 0\bmod 4$ we must
choose them inductively, beginning with $\delta_a$ and using the fact
that \divided{q}\ has at least one face of each orientation to choose
$\delta_{i+1}'$ based on the choice of $\delta_i$. The face
$\delta_{i+1}$ of $L_i$ may then be chosen arbitrarily from those
left.

By Lemma~\ref{lineardisc.lem} each face $\delta_i\cong\divided{q}$
contains a triangulated disc $D_i$ with $r\geq q$ $n$-simplices
$\Delta_{i1},\ldots,\Delta_{ir}$, such that
\[
(D_i)_{cd} = \bigcup_{k=c}^d \Delta_{ik}
\]
is a disc for each $1\leq c\leq d\leq r$. Let $\phi_i$ be the
restriction of $\psi_i$ to $D_i$, let $D_{i+1}'=\phi_i(D_i)$, and
for $1\leq j\leq r$ let $P_{ij}$ be the oriented sphere satisfying
\begin{align*}
P_{ij} \cap L_i &= \Delta_{ij}, & P_{ij} \cap L_{i+1} &= \phi_i(\Delta_{ij})
\end{align*}
that results from applying
Corollary~\ref{embeddedspheres.cor} to the pairs $(L_i,D_i)$ and
$(L_{i+1},D_{i+1}')$. 

For $1\leq k\leq r$ we now consider the sums $\sum_{j=1}^k[P_{ij}]$ modulo
$q$.
Since there are $q$ possible values mod $q$ and at least $q$ sums we
may again choose integers $c_i,d_i$ satisfying $1\leq c_i\leq d_i\leq r$ such 
that 
\[
\sum_{j=c_i}^{d_i} [P_{ij}] \equiv 0\bmod q.
\]
Let $Q_i=\sum_{j=c_i}^{d_i} P_{ij}$. Then $Q_i$ represents an $n$-sphere 
with all its vertices on $L_i\cup L_{i+1}$ and satisfying
\begin{align*}
Q_i \cap L_i &= (D_{i})_{c_id_i}, & 
Q_i \cap L_{i+1} &= \phi_i((D_{i})_{c_id_i}), &
\link{R}{Q_i}\equiv 0\bmod q.
\end{align*}

If $\link{R}{Q_i}\neq 0$ for some $i$ then we are done by setting $S=Q_i$,
so we may assume that in fact $\link{R}{Q_i}=0$ for all $i$. In that
case we let $S$ be the complex obtained from $L_a,\ldots,L_b$ and
$Q_a,\ldots,Q_{b-1}$ by omitting the interiors of the discs
$Q_a\cap {L_a},\ldots,Q_{b-1}\cap L_{b-1}$ and
$Q_a\cap {L_{a+1}},\ldots,Q_{b-1}\cap L_{b}$. Then $S$ is a
connect sum of $n$-spheres, hence an $n$-sphere, and as a chain we have
\[
S= \sum_{i=a}^b L_i + \sum_{i=a}^{b-1} Q_i.
\]
It follows that
\[
[S]= \sum_{i=a}^b [L_i] + \sum_{i=a}^{b-1} [Q_i]
   =  \sum_{i=a}^b [L_i] >0,
\]
and since also $\sum_{i=a}^b [L_i]\equiv 0\bmod q$ we are done.
\end{proof}

\begin{remark}
\label{modpbound.rem}
For $n=1$ the auxiliary sphere $\sph$
of Lemma~\ref{connectingsphere.lem} is not needed to construct
the keyring, reducing the number of vertices required in this
case to
\[
4q^2V(1,q) = 4q^2(6+15(q-1)) = 12q^2(5q-3),
\]
as given after the statement of the theorem.
\end{remark}

\section{An alternate proof of Theorem~\ref{linkingnumber.th}}
\label{linking-alternate.sec}

To further illustrate the applications of the triangulations of
Section~\ref{triangulations.sec} we give a second proof of 
Theorem~\ref{linkingnumber.th}, without the polynomial bound on the
number of vertices required. Namely, we show that given 
$\ell\in\naturalnumber$, for $N$ sufficiently large every embedding of
\complete{N}\ in \ambient\ contains a 2-component link with linking
number at least $\ell$ in absolute value. 

The proof we give is modelled on Flapan's original
proof~\cite{flapan2002} of the corresponding result for $n=1$. Her
argument is based on combining 2-component links with ``sufficiently
many vertices'', and for $n\geq 2$ we will replace this condition on
the number of vertices with a requirement that the components contain
triangulated $n$-simplices of sufficient side length. The side length
available will typically shrink when two components are combined
(unlike the number of vertices, which typically goes up), and
consequently this change leads to a significant change in the growth
of the number of vertices required.

\subsection{Splicing links}

In this section we establish higher dimensional 
analogues of Lemmas~2 and~1 of Flapan~\cite{flapan2002}. These
are Lemmas~\ref{4to3.lem} and~\ref{3to2.lem} below respectively.
In preparation for this we need an additional technical lemma
on triangulated $n$-simplices. 

\begin{lemma}
\label{sidelengthdrop.lem}
Deleting an arbitrary $M$-simplex from a triangulated $M$-simplex
of side-length $\ell$ leaves a triangulated $M$-simplex of side-length
at least $\lfloor M\ell/(M+1)\rfloor$. 
\end{lemma}

\begin{proof}
Let $\delta$ be the deleted simplex, and let $\vect{x}$ be a point 
in the interior of $\delta$. In barycentric co-ordinates on
$\Delta^M_\ell$ we have
\[
\vect{x} = \ell\sum_{i=1}^{M+1} t_i\vect{e}_i,
\]
and since $\sum t_i = 1$ we must have  $t_i\leq 1/(M+1)$ for some
$i$.  Let $\Delta$ be the intersection of $\divided[M]{\ell}$ with the
halfspace $x_i\geq \lceil \ell/(M+1)\rceil$. Then $\Delta$ is a
triangulated $M$-simplex contained in $\Delta^M_\ell$, and
$\Delta$ does not contain $\delta$ because $\Delta$
does not contain $\vect{x}$. Moreover, $\Delta$ has side-length
\[
\ell - \left\lceil\frac{\ell}{M+1}\right\rceil
=\left\lfloor\ell - \frac{\ell}{M+1}\right\rfloor
=\left\lfloor\frac{M\ell}{M+1}\right\rfloor,
\]
so we are done.
\end{proof}

\begin{lemma}
\label{4to3.lem}
Let $X_1\cup Y_1\cup X_2\cup Y_2$ be a $4$-component link contained
in some embedding of \complete{N}\ in \ambient. Suppose that for
some orientation of $X_1\cup Y_1\cup X_2\cup Y_2$ we have
$\link{X_1}{Y_1}\geq 1$ and $\link{X_2}{Y_2}= p\geq 1$, and suppose also
that each component contains a triangulated $n$-simplex of side-length
$\ell$ with $\ell^n\geq p$. Then \complete{N}\ contains disjoint
$n$-spheres $L$, $Z$ and $W$ such that
\begin{enumerate}
\item
$\link{L}{Z}=p_1\geq 1$ and $\link{L}{W}=p_2\geq p$ for some orientation of
the link $L\cup Z\cup W$;
\item
$L$ contains a triangulated $n$-simplex of side-length
at least $\lfloor n\ell/(n+1)\rfloor$;
\item
$Z$ is equal to either $X_1$ or $Y_1$;
\item
$W$ is equal to either $X_2$ or $Y_2$.
\end{enumerate}
\end{lemma}

\begin{proof}
As in Flapan~\cite{flapan2002}, if $\link{X_2}{Y_1}$ is non-zero we
may set $L=X_2$, $Z=Y_1$, and $W=Y_2$; and if $\link{Y_2}{X_1}$ is non-zero we
may set $L=Y_2$, $Z=X_1$, and $W=X_2$. So in what follows we
may assume that $\link{X_1}{Y_2}=\link{X_2}{Y_1}=0$. 

Let $D_i$ be a \divided{\ell}\ contained in $X_i$, 
for each $i$, and let $\phi:D_1\to D_2$ be a simplicial isomorphism.
After reversing orientation on both $X_1$ and $Y_1$ if necessary we
may assume that $\phi$ reverses orientation, and so we may apply
Corollary~\ref{embeddedspheres.cor} to the pairs $(X_1,D_1)$ and
$(X_2,D_2)$. We label the resulting spheres $P_0,\ldots,P_{\ell^n}$
as in the statement of the corollary, and following Flapan the
equation
\[
[X_1] + [X_2] + \sum_{j=0}^{\ell^n} [P_j]=0
\]
holds in the $n$th homology group $H_n(\ambient-Y_2;\integer)$.

By our assumption that $\link{X_1}{Y_2}=0$ we have $[X_1]=0$ in
$H_n(\ambient-Y_2;\integer)$, so 
\[
0 < p = [X_2] = - \sum_{j=0}^{\ell^n} [P_j].
\]
The right hand side consists of $\ell^n+1>p$ terms, so for some index
$q$ we must have $[P_q]\geq 0$. We consider two cases, according to
whether or not $[P_q]=0$ in $H_n(\ambient-Y_1;\integer)$.

If $[P_q]$ is non-zero in $H_n(\ambient-Y_1;\integer)$ then we construct
$L$ from $P_q$ and $X_2$ by deleting the interior of the disc
$X_2\cap P_q$. $L$ is the connect sum of the $n$-spheres $P_q$ and $X_2$,
and so is itself an $n$-sphere. As a chain we have
$L=P_q+X_2$, and therefore
\begin{align*}
[L] &= [P_q]+[X_2] \geq p   & \text{in $H_n(\ambient-Y_2 ;\integer)$}, \\
[L] &= [P_q]+[X_2] = [P_q] \neq 0   &\text{in $H_n(\ambient-Y_1 ;\integer)$}. 
\end{align*}
So we obtain the desired link by letting $Z=Y_1$ and $W=Y_2$, and
re-orienting $Z$ if necessary so that $\link{L}{Z}$ is positive.

If $[P_q]=0$ in $H_n(\ambient-Y_1;\integer)$ then we construct $L$
from $X_1$, $X_2$ and $P_q$ by deleting the interiors of the discs
$X_i\cap P_q$. Clearly, $L$ is again an $n$-sphere. As a chain we
have $L=X_1+P_q+X_2$, and therefore
\begin{align*}
[L] &= [X_1]+[P_q]+[X_2] = [P_q]+[X_2] \geq p  
           & \text{in $H_n(\ambient-Y_2 ;\integer)$}, \\
[L] &= [X_1]+[P_q]+[X_2] = [X_1] \geq 1 
       & \text{in $H_n(\ambient-Y_1 ;\integer)$}. 
\end{align*}
So we obtain the desired link by letting $Z=Y_1$ and $W=Y_2$. 

In every case above $Z$ was equal to either $X_1$ or $Y_1$, and $W$
was equal to either $X_2$ or $Y_2$.  To complete the proof we must
show that $L$ contains a triangulated $n$-simplex of side-length at
least $\lfloor n\ell/(n+1)\rfloor$.  If $q=0$ then $L$ contains $D_2$
and we are done, and otherwise $L$ contains $D_2\setminus (X_2\cap P_q)$ and we
are done by Lemma~\ref{sidelengthdrop.lem}.
\end{proof}

\begin{lemma}
\label{3to2.lem}
Let $L\cup Z\cup W$ be a $3$-component link contained in some embedding
of \complete{N}\ in \ambient, and suppose that for some
orientation of $L\cup Z\cup W$ we have   $\link{L}{Z}=p_1>0$,
$\link{L}{W}=p_2>0$. Suppose that $Z$ and $W$ contain triangulated
simplices $\Delta_Z$ and $\Delta_W$ of side-length $\ell$, with 
$\ell^n\geq p_1+p_2$, and that there is an orientation reversing
simplicial isomorphism $\phi:\Delta_Z\to\Delta_W$. Then \complete{N}\
contains an $n$-sphere $J$ disjoint from $L$ such that
\begin{enumerate}
\item
$\link{L}{J}\geq p_1+p_2$ for some orientation of $L\cup J$;
\item
\label{sidelength.item}
$J$ contains a triangulated $n$-simplex of side-length at least
$\lfloor n\ell/(n+1)\rfloor$.
\end{enumerate}
\end{lemma}

\begin{proof}
As in the proof of Lemma~\ref{4to3.lem} we apply 
Corollary~\ref{embeddedspheres.cor} to the pairs $(Z,\Delta_Z)$ and
$(W,\Delta_W)$, obtaining spheres $P_0,\ldots,P_{\ell^n}$. 
In the homology group $H_n(\ambient-L;\integer)$ we have the 
equation
\[
[Z] + [W] + \sum_{j=0}^{\ell^n} [P_j]=0,
\]
so that
\[
p_1+p_2 = [Z] + [W] = - \sum_{j=0}^{\ell^n} [P_j].
\]
As in the proof of Lemma~\ref{4to3.lem} above, the right-hand side
has $\ell^n+1>p_1+p_2$ terms, so there must be an index $q$ such that
$[P_q]\geq 0$. Let $J$ be the $n$-sphere obtained from $Z$, $P_q$ and
$W$ by deleting the interiors of the discs $P_q\cap Z$ and $P_q\cap W$.
Then $J$ is disjoint from $L$ by Remark~\ref{disjoint.rem}, and as a chain 
$J=Z+P_q+W$, so
\[
[J] = [Z]+[P_q]+[W] \geq p_1+p_2
\] 
in $H_n(\ambient-L;\integer)$. Condition~\eqref{sidelength.item} above holds
by the same argument as in Lemma~\ref{4to3.lem}, and the result
follows.
\end{proof}

Combining Lemmas~\ref{4to3.lem} and~\ref{3to2.lem} 
we obtain the following:
\begin{corollary}
\label{fourtotwo.lem}
Let $X_1\cup Y_1\cup X_2\cup Y_2$ be a $4$-component link contained
in some embedding of \complete{N}\ in \ambient. Suppose that 
\begin{enumerate}
\item
for
some orientation of $X_1\cup Y_1\cup X_2\cup Y_2$ we have
$\link{X_1}{Y_1}\geq 1$ and $\link{X_2}{Y_2}= p\geq 1$; 
\item
each component contains a triangulated $n$-simplex of side-length
$\ell$ with $\ell^n\geq 2p$;
\item
\label{orientation.item}
either $n\not\equiv 0\bmod 4$, or $X_1$ and $Y_1$ each contain 
two such triangulated $n$-simplices, one of each possible orientation.
\end{enumerate}
Then \complete{N}\ contains disjoint $n$-spheres $L$ and $J$, each
containing a triangulated $n$-simplex of side length at least $\lfloor
n\ell/(n+1)\rfloor$, and such that $\link{L}{J}\geq p+1$.
\end{corollary}

\begin{proof}
The hypotheses of Lemma~\ref{4to3.lem} are satisfied, so we
obtain a three component link $L\cup Z\cup W$ satisfying the conditions
given in that Lemma. These conditions imply the hypotheses of 
Lemma~\ref{3to2.lem}, except perhaps the condition that
$\ell^n\geq p_1+p_2$ and the condition that $\phi$ may be chosen to
reverse orientation. 

If the hypothesis $\ell^n\geq p_1+p_2$ does not hold then we must
have $p_1+p_2>2p$, which implies $p_i\geq p+1$ for some $i$. So if this
occurs we are done by simply letting $J$ be either $Z$ or $W$, as 
appropriate. 

To see that the condition on $\phi$ is satisfied we use our third
hypothesis above.
If $n\not\equiv 0\bmod 4$ then \divided{\ell}\ has an orientation 
reversing symmetry, and otherwise $Z$ is equal to either $X_1$ or $Y_1$,
and so contains a \divided{\ell}\ of each orientation. We may therefore
choose $\Delta_Z$ and $\Delta_W$ to have opposite orientations, and
apply Lemma~\ref{3to2.lem} to get the desired result.
\end{proof}

\subsection{Theorem~\ref{linkingnumber.th}, revisited}

Using the results of the previous section we re-prove 
Theorem~\ref{linkingnumber.th} in the following weakened form.

\begin{theorem}
\label{linkingnumber-alternate.th}
Given $\lambda\geq 2$, let 
$\mu = \left\lceil \sqrt[n]{2(\lambda-1)}\right\rceil$,
and suppose that $N$ is sufficiently large that \complete{N}\ contains
disjoint copies of \comdiv{2n+4}{2^i\mu}\ for $i=0,\ldots,\lambda-2$,
and an additional disjoint copy of \comdiv{2n+4}{2^{\lambda-2}\mu}.
Then every embedding
of \complete{N}\ in $\ambient$ contains a two-component link
$L\cup J$ such that, for some orientation of the components,
$\link{L}{J}\geq\lambda$.
\end{theorem}

\begin{proof}
Given an embedding of $\complete{N}$ in \ambient, 
let $C_1,\ldots,C_\lambda$ be disjoint subcomplexes of \complete{N}\
such that $C_1$ is a \comdiv{2n+4}{2^{\lambda-2}\mu}, and 
$C_i$ is a \comdiv{2n+4}{2^{\lambda-i}\mu}\ for $i=2,\ldots,\lambda$.  
Each $C_i$ is
homeomorphic to \complete{2n+4}, and so by
Taniyama~\cite{taniyama2000} contains a two component link $S_i\cup
T_i$ which we may orient such that $\link{S_i}{T_i}\geq 1$. We will
use these to inductively construct links $L_i\cup J_i$ such that
\begin{enumerate}
\item
$\link{L_i}{J_i}\geq i$;
\item
all vertices of $L_i\cup J_i$ lie in $C_1\cup\cdots\cup C_i$ (and so
$L_i\cup J_i$ is disjoint from $C_j$ for $j>i$);
\item
\label{embeddeddisc.item}
for $i<\lambda$ the spheres $L_i$ and $J_i$ each contain a
triangulated $n$-complex of side-length at least $2^{\lambda-i-1}\mu$.
\end{enumerate}
The link $L_\lambda\cup J_\lambda$ is then the required link.

Each component $S_i$, $T_i$ is isomorphic to the boundary of a
triangulated $(n+1)$-simplex of side-length equal to that of $C_i$,
and as such has $n+1$ faces which are each a triangulated $n$-simplex
of this same side-length.  For the base case we may therefore simply
let $L_1\cup J_1=S_1\cup T_1$.

Given $1\leq i\leq\lambda-1$, suppose that we have constructed
$L_i\cup J_i$ but not yet $L_{i+1}\cup J_{i+1}$.  Let
$\link{S_i}{T_i}=p\geq i$. If $p\geq \lambda$ then we simply set
$L_j\cup J_j=S_i\cup T_i$ for $j\geq i$ and the construction is
complete, so suppose that $p<\lambda$. Then every component of the
link $S_{i+1}\cup T_{i+1}\cup L_i\cup J_i$ contains a triangulated
$n$-simplex of side-length at least $\ell = 2^{\lambda-i-1}\mu\geq
\mu$, and $\ell$ satisfies $\ell^n\geq \mu^n \geq 2(\lambda-1)\geq
2p$.  Moreover, as the boundary of a $\comdiv{n+1}{\ell}$, each
component of $S_{i+1}\cup T_{i+1}$ must contain at least one
$\divided{\ell}$ face of each orientation. Working entirely within
the \complete{M}\ spanned by the vertices of $C_1\cup\cdots\cup
C_{i+1}$ we may therefore apply Corollary~\ref{fourtotwo.lem}
to obtain a $2$-component link $L_{i+1}\cup J_{i+1}$ satisfying 
$\link{L_{i+1}}{J_{i+1}}\geq p+1\geq i+1$. 

Each component of $L_{i+1}\cup J_{i+1}$ contains a triangulated
$n$-simplex of side-length at least
\[
\left\lfloor \frac{n\ell}{n+1}\right\rfloor =
\left\lfloor \frac{2^{\lambda-i-1}n\mu}{n+1}\right\rfloor.
\]
Now $\frac{n}{n+1}\geq \frac{1}{2}$, so for
$i<\lambda-1$ the quantity $2^{\lambda-i-2}\mu$ is an integer satisfying
\[
\frac{2^{\lambda-i-1}n\mu}{n+1} \geq \frac{2^{\lambda-i-1}\mu}{2}
= 2^{\lambda-i-2}\mu,
\]
and therefore 
\[
\left\lfloor \frac{n\ell}{n+1}\right\rfloor =
\left\lfloor \frac{2^{\lambda-i-1}n\mu}{n+1}\right\rfloor \geq 
2^{\lambda-i-2}\mu = 2^{\lambda-(i+1)-1}\mu. 
\]
This establishes condition~\eqref{embeddeddisc.item} above
when $i+1<\lambda$, completing
the inductive step.
\end{proof}

\bibliographystyle{plain} 
\bibliography{linking}

\end{document}